\RequirePackage{fix-cm}
\documentclass[smallextended, envcountsame]{svjour3}
\smartqed
\def\makeheadbox{\relax}
\journalname{}

\usepackage[unicode = false,
    pdftoolbar = true,
    colorlinks = true,
    linkcolor = blue,
    citecolor = blue,
    filecolor = black,
    urlcolor = blue,
    breaklinks = true]{hyperref}
\usepackage[ruled]{algorithm2e}
\usepackage{float}
\usepackage{booktabs}
\usepackage[utf8]{inputenc}
\usepackage{listings}           
\usepackage{amsmath,amssymb}            
\usepackage{xcolor}             
\usepackage{graphicx}           
\usepackage{verbatim}
\usepackage{geometry}
\usepackage{caption}
\usepackage{enumerate}
\usepackage[thinlines, thiklines]{easybmat}
\usepackage{hhline}
\usepackage{nicematrix}
\usepackage{mathtools}
\usepackage[noblocks]{authblk}
\usepackage[section]{placeins}

\DeclareMathOperator*{\argmax}{arg\,max}
\DeclareMathOperator*{\argmin}{arg\,min}

\DeclareMathOperator*{\Id}{\mathbf{I}}

\DeclareMathOperator{\spann}{span}
\DeclareMathOperator{\pspan}{pspan}

\newcommand{\R}{\mathbb{R}}

\newcommand{\cm}[1]{\operatorname{cm}\left(#1\right)}

\newcommand{\V}{\mathrm{cV}}

\newcommand{\bbm}{\begin{bmatrix}}
\newcommand{\ebm}{\end{bmatrix}}

\newcommand{\zero}{\mathbf{0}}
\newcommand{\Pe}{\mathcal{P}}

\newcommand{\vv}{\mathbf{v}}
\newcommand{\pe}{\mathbf{p}}
\newcommand{\Proj}{\operatorname{Proj}}
\newcommand{\dom}{\operatorname{dom}}
\newcommand{\intt}{\operatorname{int}}
\newcommand{\inttD}{D^\mathrm{o}}
\newcommand{\cl}{\operatorname{cl}}

\newcommand{\sgn}{\operatorname{sgn}}

\newcommand{\N}{\operatorname{\mathbb{N}}}

\newcommand{\AM}{\mathbf{A}}
\newcommand{\A}{\operatorname{\mathcal{A}}}

\newcommand{\Set}{\mathcal{S}}

\newcommand{\e}{\mathbf{e}}

\newcommand{\yy}{\mathbf{y}}

\newcommand{\x}{\mathbf{x}}
\newcommand{\y}{\mathbf{y}}

\newcommand{\one}{\mathbf{1}}

\newcommand{\Pee}{\mathbb{P}}
\newcommand{\pee}{\mathbf{p}}
\newcommand{\ee}{\mathbf{e}}

\newcommand{\ess}{\mathcal{S}}

\newcommand{\tee}{\mathcal{T}}

\newcommand{\pp}{\mathcal{P}}

\newcommand{\bpspan}{\delta\pspan(\ess_i)}

\newcommand{\Sn}{\mathbb{S}^n}
\DeclareMathOperator{\sm}{sm}

\newcommand{\dd}{\mathbf{d}}
\newcommand{\uu}{\mathbf{u}}

\DeclareMathOperator{\CM}{cm}
\DeclareMathOperator{\npspan}{\overline{pspan}}

\DeclareMathOperator{\subjectto}{subject \, to}
\DeclareMathOperator{\nul}{null}

\DeclareMathOperator{\inte}{int}

\newcommand{\cmf}{\CM(f;\x)}

\DeclareMathOperator{\polar}{polar}
\DeclareMathOperator{\CV}{cV}

\newcommand{\ph}{\phantom{-}}

\newcommand{\ssig}{\boldsymbol{\sigma}}

\newcommand{\xo}{\mathbf{x}_0}
\newcommand{\bfs}{\mathbf{s}}

\begin{document}

\title{\centering The cosine measure of a function at a point}
\author{\centering  Warren Hare \and Gabriel Jarry-Bolduc \and Chayne Planiden}
\institute{
Warren Hare \at
              Department of Mathematics, University of British Columbia, Kelowna, 
              British Columbia, Canada. \\ Hare's research is partially funded by the Natural Sciences and Engineering Research Council (NSERC) of Canada, Discovery Grant \#2023-03555. ORCID 0000-0002-4240-3903\\
              \email{warren.hare@ubc.ca}           
           \and
          Gabriel Jarry-Bolduc\at
          Department of Mathematics and Statistics, American University of Sharjah, Sharjah, United Arab Emirates. \\ ORCID 0000-0002-1827-8508\\
              \email{gabjarry@alumni.ubc.ca} 
              \and
               Chayne Planiden   \at
               School of Mathematics and Physics, University of Wollongong, Wollongong, NSW, 2500, Australia. \\ Research supported by University of Wollongong. ORCID 0000-0002-0412-8445 \\ \email chayne@uow.edu.au 
              }
\date{\today}

\titlerunning{Cosine measure of a function}

\maketitle
\begin{abstract}
The cosine measure of a set of vectors in $\R^n$ measures how well the set covers all directions in $\R^n$. It identifies the direction furthest, in angle, from the set. It is used in the convergence theory of various optimization algorithms, but also highlights interesting geometric properties of sets.  For example, the cosine measure of a set $S$ is greater than zero if and only if given any $\mathcal{C}^1$ function $f$ at a point $\bold{x}$ where the gradient is nonzero, $S$ must contain a descent direction of $f$ at $\bold{x}$.  In this paper, we examine the question of what can be said when the function $f$ is non-differentiable or if it has a gradient equal to the zero vector.  To examine these cases, we introduce the novel concept of the {\em cosine measure of a function} at a point.  This value provides an infimum on the value of the cosine measure that a set of vectors requires to guarantee it contains a descent direction of the function at the point of interest.  We present mathematical theory around this concept, including examples showing that the cosine measure of a smooth function can have any value in $[-1,1]$.  We further present algorithms to compute the cosine measure of a function, and  examples demonstrating the algorithm on smooth and nonsmooth functions.  These results also shed light on the the cosine measure of infinite sets and nonconvex cones.
\end{abstract}

\keywords{Cosine measure \and Descent directions \and Infinite sets \and Finite-max function \and $\ell_1$ norm}
\section{Introduction}\label{sec:intro}
In \emph{Derivative-free optimization} (DFO), \emph{direct search methods} usually rely on \emph{positive spanning sets} to explore the space around an incumbent solution. A notion closely related to positive spanning sets is the concept of \emph{positive bases}. A positive basis can be thought of as a minimal positive spanning set. In other words, if any one vector in a positive basis is removed, then it is no longer a positive spanning set. The properties of positive spanning sets and positive bases have been studied in \cite{Audet2011,Conn2009,cornaz2024characterization,Davis1954,Hare2017,Hare2016cardinality,Jarry2023thesis,RMLewis_VTorczon_1996,Reay1966,Regis2016,Romanowicz1987,Shephard1971}. In \cite{hare2024using,marcus1981minimal,marcus1984gale}, the notion of \emph{positive $k$-spanning set} is studied and its value in DFO is briefly discussed. 

One property of positive spanning sets (that make them useful in DFO methods) is that if the gradient of a function at a point is well-defined and not equal to the zero vector, any positive spanning set contains at least one descent direction (and at least one ascent direction) for the function at the point. To decide whether a given finite set of vectors is a positive spanning set of $\R^n$, the \emph{cosine measure} can be computed \cite{kolda2003}. It is known that a set of vectors is a positive spanning set of $\R^n$ if and only if the cosine measure is greater than zero \cite{audet2024cosine}.   The cosine measure is  a real number between $-1$ and $1$ inclusively. Its value quantifies the largest angle that can be created between the set and an arbitrary vector in $\R^n$; with values near $1$ representing that this largest angle is close to $0$ (i.e., $\cos(\theta) \approx 1$).  The structure of positive bases with maximal cosine measure is investigated in \cite{hare2023nicely,Jarry2019,naevdal2018}. The convergence theory supporting a direct search method usually assumes that the cosine measure values of the positive spanning sets utilized are bounded away from zero \cite{conn2009introduction}. 
The cosine measure is also used to define the notion of \emph{complexity measure} that arises in complexity bounds for direct search methods \cite{cavarretta2025,Dodangeh2016}.

The main goal  of  this paper  is to investigate the value of the cosine measure necessary to guarantee that a set of vectors contains a descent direction of a real-valued function $f$. When the gradient of $f$ at a point is well-defined and not equal to the zero vector, then a set of vectors with cosine measure value greater than zero guarantees that the set contains a descent direction of $f$ at the point. This is not necessarily true if the gradient is equal to the zero vector or if the gradient does not  exist. The notion of \emph{cosine measure of a function} emerges from the investigation of the two latter cases.  We present theoretical results and examples exploring this notion.

A deterministic algorithm for computing the cosine measure of a finite positive spanning set was proposed in \cite{hare2020deterministic}. A few months later, an alternative approach to compute the cosine measure of a finite set of vectors was introduced in \cite{regis2021}.  Further research has revealed that computing the cosine measure exactly is an NP-hard problem \cite{sphericalDiscJones2020}, and proposed heuristics to compute the value quickly in high dimension \cite{HareSun2025}.  Recently, the notion of cosine measure has been extended to consider \emph{cosine measure relative to a subspace} and an algorithm for its computation is presented in \cite{audet2024cosine}.  All of the aforementioned algorithms begin with the assumption that the set of vectors is finite.  

A secondary result of this current work is the development of algorithms to determine cosine measures where the set contains infinitely many vectors.  This leads to algorithms for computing the cosine measure of a classical type of nonsmooth function: the finite-max function. In this case, the Clarke subdifferential at the point of interest is the 
convex hull of the active gradients, in line with the 
general gradient-limit representation of Clarke subdifferentials for 
stratifiable functions established in \cite{drusvyatskiy2015clarke}. The 
set of non-descent directions for this type of function at the point of interest is  the polar of 
the cone generated by the active gradients.  The cosine measure of the function 
at a certain point quantifies the angular size of this cone, measuring how difficult 
it is for a finite set of directions to contain a descent direction. 
Whereas gradient-sampling methods \cite{burke2002approximating} address 
this difficulty by approximating the Clarke subdifferential through 
randomly sampled gradients, the approach in this paper is deterministic. 

The remainder of this paper is organized as follows. In Section \ref{sec:prel}, the notation is clarified, and background results and  necessary definitions  to understand the remainder of this paper are provided. In Section \ref{sec:cosineMeasureOfAFunction}, the concept of cosine measure of a function is defined and properties of this novel definition are introduced. Examples are provided showing  the cosine measure of a function can take any value between $-1$ and 1 inclusively when the gradient is equal to the zero vector.  In Section \ref{sec:algo}, pseudo-codes to compute the cosine measure of infinite sets are introduced. In Section \ref{sec:cmNonSmooth}, we show how to find the value of the cosine measure of a finite-max function. In particular,  a closed-form solution is provided for the $\ell_1$-norm function. Lastly, the main results of this paper are summarized, and future research directions are proposed in Section \ref{sec:conclusion}.

\section{Preliminaries} \label{sec:prel}

In this paper, we follow the standard notation found in \cite{rockwets}. A vector is written as a column vector and denoted with a bold-face lower-case letter. A matrix is denoted with a bold-face capital letter.  The transpose of matrix $\AM$ is denoted by $\AM^\top$.  We work in finite-dimensional space $\R^n$ with inner product $\x^\top \y=\sum_{i=1}^nx_iy_i$ and induced norm $\|\x\|=\sqrt{\x^\top \x}$. The identity matrix in $\R^{n \times n}$ is denoted by $\Id_n$. We use $\e_i \in \R^n$ where $i \in \{1, 2, \dots, n\},$ to denote the standard unit basis vectors in $\R^n$, i.e.\ the $i$\textsuperscript{th} column of $\Id_n.$ The zero vector in $\R^n$ is denoted by $\zero_n$ and the vector of all ones in $\R^n$ is denoted by $\one_n.$ When there is no ambiguity about the dimensions of the matrices or vectors considered, the subscript $n$ referring to the dimension may be dropped.

The notation  $f:\dom f \subseteq \R^n \to \R$ is used to refer to a real-valued function $f$ where the codomain is $\R$ and $\dom f$ is assumed to be a non-empty subset of $\R^n$ where $f$ is well-defined. The interior of a set $\Set$ is denoted by $\intt \Set$ and the closure of the set $\Set$ is denoted by $\cl \Set.$ The column space (span) of a set $\Set$  is denoted by $\spann(\Set)$. the set of all unit vectors in $\R^n$ is denoted by $\Sn.$ The definition of positive span follows.
\begin{definition}[Positive span]  Let $\Set$ be a non-empty set in $\R^n$. 
     The {\em positive span} of $\Set$ is denoted by $\pspan(\Set)$ and defined by
    $$\pspan(\Set) = \left\{\x \in \R^n: \x =\sum_{i=1}^k \lambda_i \dd_i, k \in \N, \dd_i \in \Set, \lambda_i \geq 0\right\}.$$
 The {\em normalized positive span} of $\Set$ is denoted by $\npspan(\Set)$ and defined by $\npspan(\Set)=\pspan(\Set) \; \cap \; \Sn.$
\end{definition}

Note that the positive span of $\Set$ is equivalent to the \emph{conical hull} of $\Set.$ When $\Set$ is finite,  $\pspan(\Set)$ is a \emph{polyhedral cone}.  The set $\Set$ is said to be a \emph{positive spanning set of $\R^n,$} or to \emph{positively span $\R^n$,} if and only if $\pspan(\Set)=\R^n.$ It is known that the cardinality of a positive spanning set of $\R^n$ is equal to or greater than $n+1.$  A finite positive spanning set of $\R^n$ with cardinality $s$ will be denoted by $\Pee_{n,s}.$ 

The following definition of cosine measure makes it possible to  consider infinite sets of vectors. The max operator in the original definition (see \cite{kolda2003}) is replaced by a supremum operator. 

\begin{definition}[Cosine measure] \cite{audet2024cosine}\label{def:cosinemeasure}  Let $\Set \subseteq \R^n$ be a non-empty (possibly infinite) set of unit vectors.
The \emph{cosine measure of $\Set$} is defined by 
    $$\cm{\Set} \ =\ \min_{\substack{\uu \, \in \, \R^n\\\Vert \uu \Vert=1}} \sup_{\dd \, \in \, \Set} \uu^\top \dd,$$
and the  \emph{cosine vector set of $\Set$}, denoted by $\V(\Set)$, is defined by $$\V(\Set) \ =\ \argmin_{\substack{\uu \, \in \, \R^n\\ \Vert \uu \Vert=1}}\sup_{\dd\, \in\, \Set} \uu^\top \dd.$$
In the case where $\Set$ is empty, we define the cosine measure of $\Set$ to be equal to $-1$ and the cosine vector set of $\ess$ to be the empty set. 
\end{definition}

\begin{definition}[active set]\label{def:allactivity} 
Let 
    $\Set \subseteq \R^n$ be a non-empty set of unit  vectors.

The \emph{active set for $\Set$ at a cosine vector $\uu \in \V(\Set)$} is denoted by  $\A(\Set,\uu)$ and defined by $$\A(\Set,\uu)\ =\ \left \{\dd \in \Set\,:\, \dd^\top\uu=\CM(\Set) \right \}.$$
\end{definition}

When $\Set$ is finite and $\pspan(\Set)=\R^n$, it is known that $\A(\Set,\uu)$ contains a basis of $\R^n$ for any $\uu \in \CV(\Set)~\cite[Cor. 18]{hare2020deterministic}.$
The following result clarifies the relation between cosine measures and positive spanning sets of $\R^n.$
\begin{proposition} \label{prop:cosinemeasurepos}
Let $\Set$ be a non-empty (possibly infinite) set of unit vectors. Then $\CM(\Set)>0$ if and only if $\Set$ is a positive spanning set of $\R^n$.  
\end{proposition}
\begin{proof} The proof follows immediately from \cite[Proposition 18]{audet2024cosine}. \qed
\end{proof}

Next, we recall the definition of descent direction. 

\begin{definition}[Descent direction]
Let $f:\dom f \subseteq \R^n \to \R$ and $\x \in \dom f.$ A unitary vector $\dd \in \R^n$  is said to be a \emph{descent direction} of $f$  at $\x$  if and only if there exists a scalar $\overline{t}>0$ such that $
f(\x+t\dd)<f(\x) \quad \text{for all $t \in (0, \overline{t}]$}.$   
\end{definition}

As mentioned in Section \ref{sec:intro}, an important property  of  positive spanning sets is that if $\nabla f$ is well-defined at  $\x \in \dom f$ and  $\nabla f(\x) \neq \zero_n,$ we  have $-\nabla f(\x)^\top \dd >0$ 
for at least one vector $\dd \in \Pee_{n,s}.$ In other words, there exists at least one descent direction of $f$ at $\x$  in a positive spanning set of $\R^n$ \cite{conn2009introduction}.
The statement does not extend to non-differentiable functions or points where the gradient is equal to the zero vector, as the next two examples illustrate. 

\begin{example} \label{ex:pssfailsone}
Consider the function $f(\x)=\max\{\vert x_1 \vert, \vert x_2 \vert \}$   at the point $\xo=\bbm 1 &1\ebm^\top$ and the set  $\Pee_{2,4}=\left \{\e_1,\e_2,-\e_1,-\e_2 \right \}.$ Then $\Pee_{2,4}$ is a positive spanning set of $\R^n$ with $\CM(\Pee_{2,4})= 1/\sqrt{2}.$ However, $\Pee_{2,4}$ does not contain a descent direction of $f$ at $\xo.$
\end{example}

\begin{example} \label{ex:pssfailstwo}
Consider the function $f(\x)=(x_1-x_2)^2-(x_1+x_2)^2$ at the point $\xo=\zero$ and the set  $\Pee_{2,4}=\left \{\e_1,\e_2,-\e_1,-\e_2 \right \}.$ Then, $\Pee_{2,4}$ is a positive spanning set of $\R^n$ with $\CM(\Pee_{2,4})= 1/\sqrt{2}.$ However, $\Pee_{2,4}$ does not contain a descent direction of $f$ at $\xo.$
\end{example}

Notice that in Examples \ref{ex:pssfailsone} and \ref{ex:pssfailstwo}, if a set of vectors   with  cosine measure strictly greater than  $1/\sqrt{2}$ is utilized, then the set will contain a descent direction of $f$ at the given respective point. We conclude  that a ``sufficiently good'' set of vectors  in  terms of cosine  measure may be able to find a descent direction in  the non-differentiable case or when the gradient is equal to the zero vector.  The question of how to quantify ``sufficiently good'' leads to the definition of  \emph{cosine measure of a function}.    

\section{Cosine measure of a function} \label{sec:cosineMeasureOfAFunction}

In this section, we introduce the definition of cosine measure of a function. We will see that it provides information on the quality that a set of vectors must have to ensure a descent direction of $f$ is found at $\x$. We investigate the properties of the cosine measure of a function, and provide examples in $\R^2$ that show the cosine measure of a function can take any value between -1 and 1 (inclusive) when the gradient is equal to the zero vector or does not exist.  

Henceforth, the set of all unit vectors in $\R^n$ is denoted by $\Sn.$ The set of all unit descent directions of $f$ at $\x$ is denoted by $D(f;\x).$ The complement of $D(f;\x)$ in $\Sn$ is denoted by $D^c(f;\x)$, i.e., 
    $$\begin{array}{rcl}
       D(f;\x)  &=& \{ \dd \in \Sn : \dd \mbox{ is a descent direction of } f \mbox{ at } \x \}, \\
       D^c(f;\x) &=& \Sn \setminus D(f;\x).
    \end{array}$$
We simply write $D$ and $D^c$ when it is clear from the context what the function $f$ and the point of interest $\x$ are.

\begin{definition}[Cosine measure and cosine vector set of $f$ at $\x$] \label{def:cmfunction}
Let $f:\dom f \subseteq \R^n \to \R$ and $\x \in \dom f.$ The cosine measure of $f$ at $\x$  is denoted by $\cm{f;\x}$ and defined by
$$\cm{f;\x}=\min_{\substack{\uu \, \in \, \R^n\\\Vert \uu \Vert=1}} \sup_{\dd \in D^c} \uu^\top \dd$$
provided $D^c \neq \emptyset.$ If $D^c= \emptyset,$ then $\cm{f;\x}$ is defined to be $-1$.\\
The cosine vector set of $f$ at $\x$ is denoted by $\V(f;\x)$  and defined by  $$\V(f;\x)=\argmin_{\substack{\uu \, \in \, \R^n\\\Vert \uu \Vert=1}} \sup_{\dd \in D^c} \uu^\top \dd$$ provided $D^c \neq \emptyset.$ If $D^c=\emptyset,$ then $\V(f;\x)$ is defined to be $\emptyset.$
\end{definition}

In the previous definition,  note that the supremum operator may be replaced by a max operator  by considering the set $\cl D^c$ rather than $D^c.$  It follows that when $D^c \neq \emptyset,$
$$\cm{f;\x}=\min_{\substack{\uu \, \in \, \R^n\\\Vert \uu \Vert=1}} \max_{\dd \in \cl D^c} \uu^\top \dd \quad \text{and} \quad \V(f;x)=\argmin_{\substack{\uu \, \in \, \R^n\\\Vert \uu \Vert=1}} \max_{\dd \in \cl D^c} \uu^\top \dd.$$

  Note that $\V(f;\x)$ is non-empty if and only if  $D^c(f;\x)$ is non-empty.
 Observe that, when $D^c \neq \emptyset$, the definition of the cosine measure of $f$ at $\x$ can be reformulated in terms of  the definition of cosine measure of a set: $$\cmf=\cm{D^c}=\cm{\cl D^c}.$$  In words, the cosine measure of $f$ at $\x$ is equal to the cosine measure of the set of all unit non-descent directions of $f$ at $\x.$ 

Theorem \ref{thm:valueofcmf} clarifies the relation between the cosine measure of a function at a point and the cosine measure of a set of vectors. First, we recall a result about the effect of an orthogonal matrix  on a set of unit vectors.

\begin{lemma}\label{lem:orthomatrix} 
    Let $\ess \subseteq \R^n$ be a closed set of  unit vectors and $R \in \R^{n \times n}$ be an orthogonal matrix.  Define the closed set $\ess_R = \{ R\dd : \dd \in \ess \}$.
    Then $\CM(\ess)=\CM(\ess_R).$
\end{lemma}
\begin{proof} This is an immediate consequence of \cite[Thm 3.4]{regis2021}.

\qed
\end{proof}

\begin{theorem}\label{thm:valueofcmf}
Let $f:\dom f \subseteq \R^n \to \R$ and $\x  \in \dom f$ and let $\Set \subseteq \R^n$ be a closed set of unit vectors (possibly infinite). If $\CM(\Set)>\cm{f;\x},$ then $\Set$ contains a descent direction of $f$ at $\x.$ Moreover, $\Set_R = \{ R\dd : \dd \in \Set\}$ contains a descent direction of $f$ at $\x$ for any orthogonal matrix $R \in \R^{n \times n}.$
\end{theorem}
\begin{proof}
 Suppose $\CM(\Set)>\CM(f;\x) \geq -1.$  Since $\CM(\Set)>-1,$ the set $\Set$ contains at least two vectors (so $\Set \neq \emptyset$). 
 
 If  $D^c$ is empty, then $D=\Sn$ and it follows that all of the vectors in $\Set$ are descent directions of $f$ at $\x$.
 
 Suppose $D^c$ is non-empty. Let $\uu \in \V(f;\x).$   
We have  $$\sup_{\vv \in \, \Set} \uu^\top \vv \geq 
 \CM(\Set)>\sup_{\dd \in D^c} \uu^\top \dd=\CM(f;\x).$$
Let $\vv^* \in \Set$ be a  vector attaining the supremum $\sup_{\vv \in \Set} \uu^\top \vv.$ Note that such a vector exists in $\Set$, since $\Set$ is a closed set. If $\vv^* \in D^c$, then 
$$\CM(f;\x)=\sup_{\dd \in D^c} \uu^\top \dd \geq \uu^\top \vv^* \geq \CM(\Set),$$ which is a contradiction to the assumption $\CM(\Set)>\cmf.$
Therefore, $\vv^* \in D$, so is a descent direction of $f$ at $\x.$ 

The second results follows from Lemma \ref{lem:orthomatrix}.
\qed
\end{proof}

It is tempting to believe that given a set $\Set$, function $f$, and point $\x$; if $\Set_R$ contains a descent direction for any orthogonal matrix $R,$ then $\CM(\Set) \geq \CM(f;\x).$ But this is not necessarily true, as the following example illustrates.
\begin{example}\label{ex:threewilldo}
Consider the function $f(\x)=x_1^2-x_2^2$ at $\xo=\zero,$  and the minimal (optimal) positive basis $\Set=\left \{ \bbm 0\\ 1\ebm, \bbm \phantom{-}\frac{\sqrt{3}}{2}\\-\frac{1}{2} \ebm, \bbm -\frac{\sqrt{3}}{2}\\-\frac{1}{2} \ebm \right \}.$ Note that $\Set_R$ contains a descent direction of $f$ at  $\xo$ for any orthogonal matrix $R \in \R^{2 \times 2}.$ (This is visualized in Figure \ref{fig:PSSfail2}.) The cosine measure of $\Set$ is $\CM(\Set)=1/2,$ but $\CM(f;\zero)=1/\sqrt{2}.$
\begin{figure}[h]
\centering
\includegraphics[scale=0.4]{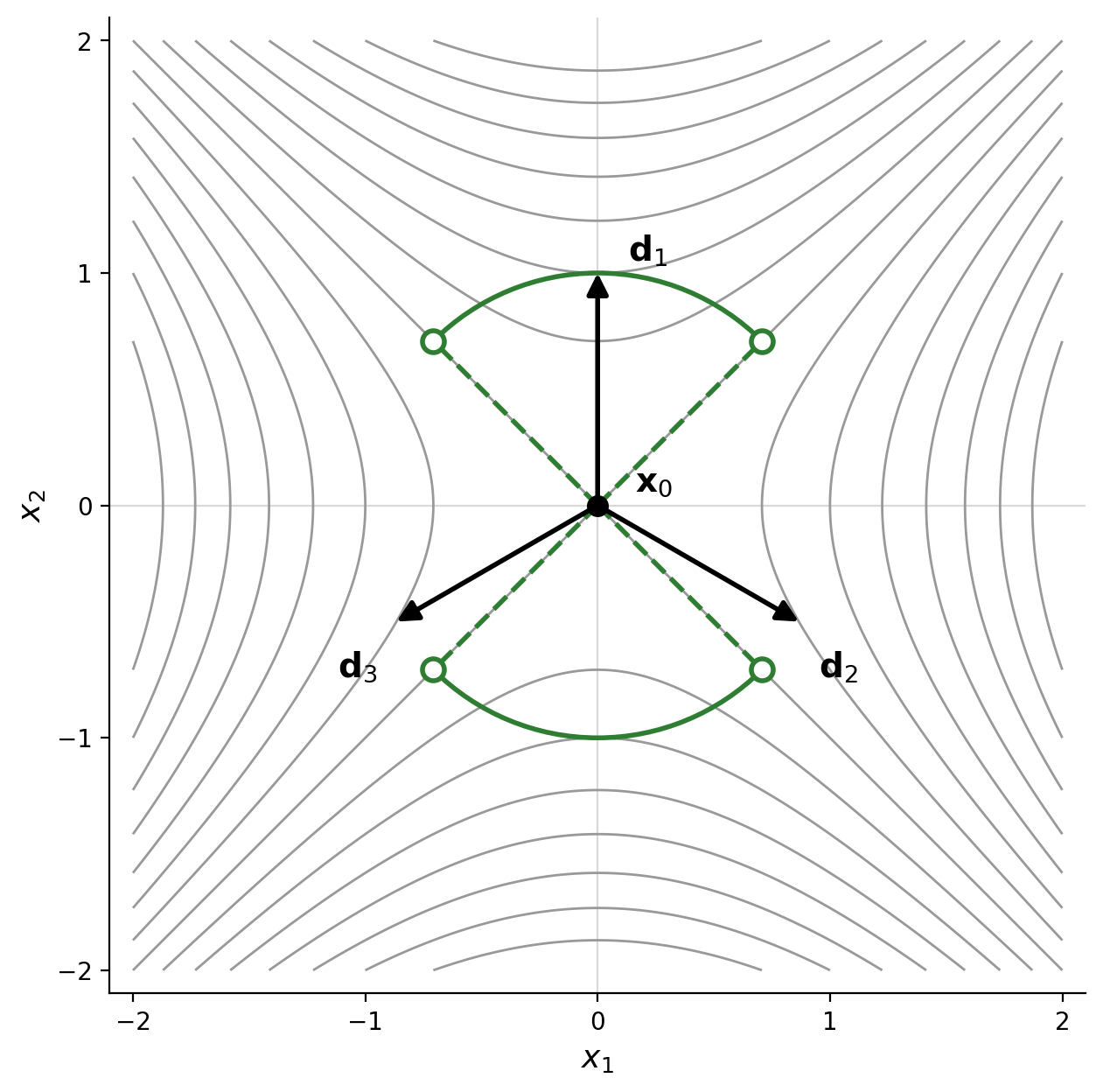}
\captionof{figure}{Any rotation or reflection of the set $\ess$ contains a descent direction of $f$ at $\xo=\zero$, but $\CM(\Set) = 1/2 < 1/\sqrt{2} = \CM(f;\zero)$. The two green spherical caps represent the descent directions of $f$ at 
$\x_0$.}
\label{fig:PSSfail2}
\end{figure} 
\end{example}

From the previous example, we conclude that it is possible to create a set $\Set$ with lower cosine measure than $\cmf$ and identify a descent direction even when the set is rotated. However, it is also possible to create a set $\Set$ with cosine measure equal to $\cmf$ such that $\Set$ does not contain a descent direction of $f$ at $\x$, provided $\cmf >-1$.

\begin{proposition}
Let $f:\dom f \subseteq \R^n \to \R$ and $\x \in \dom f.$  $Let \cmf>-1$. Then there exists a non-empty set of vectors $\Set$ such that $\CM(\Set)=\cmf$ and $\Set$ does not contain a descent direction of $f$ at $\x.$
\end{proposition}
\begin{proof} Take  $\ess=D^c.$ \qed  
\end{proof}


We now provide several results about the value of the cosine measure of $f$ at $\x.$ In examining Theorem \ref{thm:maintheorem}(iii), recall that, assuming $\pspan(\cl \Set) \neq \R^n,$ we have $\CM(\cl \Set)=0$ if and only if  $\cl \Set$ contains a non-empty proper subset $V$ such that $\pspan(V)=\spann(V)$ \cite[Corollary 29]{audet2024cosine}. This implies that $\cl \Set$ contains a positive spanning set of a linear subspace $L\subset \R^n$ with $1 \leq \dim(L) <n.$

\begin{theorem} \label{thm:maintheorem}
Let $f:\dom f \subseteq \R^n \to \R$  and  $\x \in \dom f.$  Then the  following hold:  
\begin{enumerate}[(i)]
\item \label{thm:CfxEqualMinus1} $\cmf=-1$ if and only if  $\vert D^c \vert \leq 1,$
\item \label{prop:closedhalfspace} $\cmf \leq 0$ if and only if   $D^c$  is contained in a closed half-space of $\R^n$ (whose boundary contains the origin),
\item \label{prop:minimalposb} if $\pspan(\cl D^c) \neq \R^n$, then $\cmf=0$ if and only if  $\cl D^c$ contains a positive spanning set of  a linear subspace $L\subset \R^n$ where $1\leq \dim(L) < n,$ 
\item \label{prop:CosineConstantEqualOne}  $\cm{f;\x}=1$ if and only if $\cl D^c=\Sn.$
\end{enumerate}
\end{theorem}
\begin{proof}
$(i)$  Suppose $\cmf=-1$ and  $\vert D^c \vert >1.$ Let $\dd_1$  and $\dd_2$ be two  distinct unit vectors in $D^c.$ Then there exists a unit vector $\uu \in \R^n$ such that $\sup_{\dd \in D^c}  \uu^\top \dd=-1.$ This means that    $\uu^\top \dd_1=\uu^\top \dd_2=-1.$  This implies  $\dd_1=-\uu$ and $\dd_2=-\uu,$ a contradiction. 

Conversely, suppose  $\vert D^c \vert \leq 1.$ If  $\vert D^c \vert=0$, then the result holds by definition. If $\vert D^c \vert=1,$ then  take  $\uu=-\dd$ where $\dd$ is the vector in $D^c.$\smallskip

\noindent $(ii)$  Suppose $\cmf \leq 0.$ If $\V(f;\x)$ is empty, then  this means $D^c=\emptyset$ and the result holds. Suppose $\uu \in \V(f;\x).$ Then $$\cmf=\sup_{\dd \in D^c} \uu^\top \dd \leq 0.$$ Let $H$ be the hyperplane in $\R^{n}$  with normal vector $\uu$ passing through the origin. Then  $D^c$  is contained in the closed half space $H^-=\{\vv \in \R^n:\uu^\top \vv\leq 0\}.$

Conversely, suppose $D^c$ is contained in a closed half-space  defined by a hyperplane $H$. Let $\uu$  be a unit normal vector to the hyperplane $H$  such that  $$\sup_{\dd \in D^c} \uu^\top \dd \leq 0.$$ Therefore, we must have $\cmf \leq \sup_{\dd \in D^c} \uu^\top \dd \leq 0.$ \smallskip

\noindent $(iii)$  The result follows immediately from  \cite[Corollary 29]{audet2024cosine}.\smallskip

\noindent $(iv)$ Suppose $\cmf=1$ and $\cl D^c  \neq \Sn.$ This means there exists a unit vector $\vv$ such that $\sup_{\dd \in D^c} \vv^\top \dd=\max_{\dd \in \cl D^c} \vv^\top \dd<1$, a contradiction. 

Conversely, suppose  $\cl D^c =\Sn.$ Let $\uu \in \V(f;\x).$  Then $\uu \in \cl D^c .$   Hence, 
$$\cmf=\sup_{\dd \in D^c} \uu^\top \dd=\max_{\dd \in \cl D^c} \uu^\top \dd=\uu^\top \uu=1.$$ \qed
\end{proof}
 
When $f$ is  a single-variable function there are only three possibilities.
\begin{enumerate}[(i)]
\item The set $D^c$ is empty. In this case  $\cmf=-1.$
\item The set $D^c$ contains one direction ($D^c =\{1\}$ or $D^c =\{-1\}$). In this case $\cmf=-1.$
\item The set $D^c$ contains two directions ($D^c = \{-1, 1\}$). In this case $\cmf=1.$
\end{enumerate}
Theorem \ref{thm:maintheorem} can be reformulated as follows for a single-variable function:
\begin{enumerate}[(i)]
\item $\cmf=-1$ if and only if $ \cl D^c=D^c  \neq \mathbb{S}^1,$
\item $\cmf=1$ if and only if  $\cl D^c =D^c=\mathbb{S}^1.$
\end{enumerate}

The next theorem provides a result about the value of the cosine measure of a function at a differentiable point where $\nabla f(\x)\neq \zero_n, n \geq 2.$
\begin{theorem}[Cosine measure at a differentiable point where $\nabla f(\x) \neq \zero$] \label{thm:cmdiffpoint}
Let $f:\dom f \subseteq \R^n \to \R$   with $n\geq 2.$ Suppose $f$  is differentiable at $\x \in \inte \dom f.$ If $\nabla f(\x)\neq \zero,$ then $\cmf=0$ and $\V(f;\x)=\left \{ -\frac{\nabla f (\x)}{\Vert \nabla f(\x)\Vert}\right \}.$  
\end{theorem}
\begin{proof}
Suppose $\nabla f(\x) \neq \zero.$ Consider the hyperplane passing through the origin and  defined by $H=\{\vv \in \R^n:-\nabla f(\x)^\top \vv=0\}.$ Then the unit vectors in $H$ with initial point at the origin  are  in $\cl D^c.$ This means that $\cl D^c$ contains a  positive spanning set of  a linear subspace with dimension $n-1$. Since $\pspan(\cl D^c)$ is not equal to $\R^n$, $\cmf=0$ by Theorem \ref{thm:maintheorem}.  Since the unit vector  $\uu \in \R^n$ such that $\uu^\top \dd \leq 0$  for all $\dd \in \cl D^c$ is unique, we  get $\V(f;\x)= \left \{-\frac{\nabla f (\x)}{\Vert \nabla f(\x)\Vert} \right \}.$   \qed
\end{proof}

Note that the previous theorem assumes $n \geq 2.$  For a single-variable function,  $\cmf=-1$ whenever $f'(x) \neq 0.$ 

The converse of Theorem \ref{thm:cmdiffpoint} is not necessarily true. Examples where $\cmf=0$ and $\nabla f(\x)=\zero$ can be created. The following section provides  examples  for the case where the gradient is equal to  the zero vector. We will see that anything can happen when $\nabla f(\x) = \zero$ (or when  $f$ is not differentiable at $\x$). 

\subsection{Cosine measure of a function when the gradient is equal to the zero vector} 
\label{sec:exampleGradEqualZero}
In this section, we provide examples in $\R^2$ showing that the cosine measure of a function may take any value between -1 and 1 inclusively. The value of the cosine measure of a function is calculated in each example.   Note that the set of non-descent directions  of a function at a point $\x,$ $D^c(f;\x)$, may not be a finite set of directions.  In \cite{audet2024cosine}, a deterministic algorithm is provided to compute the cosine measure (relative to a subspace) of any non-empty finite set of nonzero vectors. Hence,  Algorithm 1 from \cite{audet2024cosine} can be applied directly to compute the cosine measure of a function  when $D^c$ is finite. In Section \ref{sec:algo}, we show how to compute the cosine measure  of a real-valued function in $\R^n.$  Alternately, by switching to polar coordinates, computing the cosine measure in $\R^2$ can be reduced to single-variable optimization.  In all of the following examples this computation is easy, so omitted.

We begin with an example where $\nabla f(\xo)=\zero$ and $\cm{f;\xo}=-1$. 

\begin{example}[$\cm{f;\xo}=-1$] Consider the function $f:\dom f=\R^2 \to \R:\x=\bbm x_1&x_2\ebm^\top\mapsto -x_1^2-x_2^2$ at the point $\xo=\zero.$ Note that $D^c$ is empty. Therefore, $\CM(f;\xo)=-1.$
\end{example}

The next example illustrates $\nabla f(\xo)=\zero$ with $\cm{f;\xo} \in (-1,0)$. 

\begin{example}[$\cm{f;\xo} \in (-1,0)$]
Consider the class of  functions $$f_{\alpha}:\dom f=\R^2 \to \R:\x \mapsto \begin{cases}  -\alpha x_1^2-x_2^2, & x_2 \leq 0\\ -\alpha x_1^2+x_2^2, & \text{otherwise,}\end{cases}$$ where $\alpha \in (0,\infty)$, at the point $\xo=\zero.$ 

\begin{figure}[H]
\centering
  \includegraphics[width=.95\linewidth]{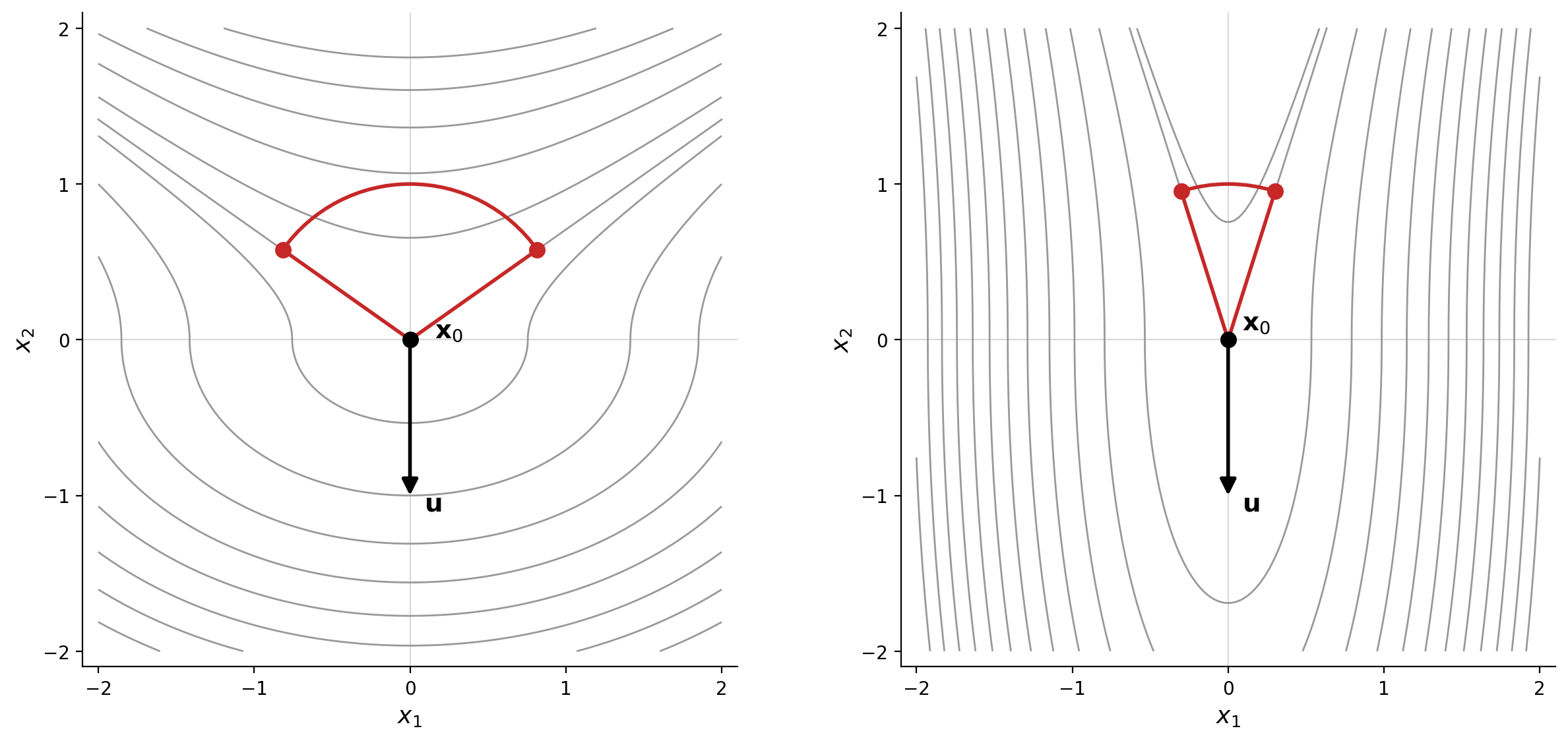}
  \caption{The set $D^c$ for $f_\alpha$  at $\xo=\zero$ is represented by the red spherical cap. On the left, $\alpha=0.5,$ and on the right $\alpha=10$.}
  \label{fig:test1}
\end{figure}

Note that $f_\alpha$ is differentiable at $\xo=\zero.$ The contour lines at $f(\x)=0$ are given by the equations $x_2=\sqrt{\alpha} \vert x_1 \vert.$ As a vector-valued function, the contour lines at $f(\x)=0$ are $L_1(t)=\bbm \frac{1}{\sqrt{1+\alpha}}&\sqrt{\frac{\alpha}{1+\alpha}}\ebm^\top t$ and $L_2(t)=\bbm \frac{-1}{\sqrt{1+\alpha}}&\sqrt{\frac{\alpha}{1+\alpha}}\ebm^\top t,$ where $t \geq 0.$  \\
 We find that the set $D^c$ is  $\npspan \left \{\bbm \frac{1}{\sqrt{1+\alpha}}&\sqrt{\frac{\alpha}{1+\alpha}}\ebm^\top, \bbm \frac{-1}{\sqrt{1+\alpha}}&\sqrt{\frac{\alpha}{1+\alpha}}\ebm^\top \right \}.$   We obtain
\begin{align*}
    \cm{f_{\alpha};\zero}=-\sqrt{\frac{\alpha}{1+\alpha}}, \quad \text{and the cosine vector is} \, \uu=\bbm 0&-1\ebm^\top.
\end{align*}
As a function of $\alpha,$ $\cm{f_\alpha;\zero}$ is a continuous function on $(0, \infty)$ with $$\lim_{\alpha \to 0^+} -\sqrt{\frac{\alpha}{1+\alpha}}=0 \quad \text{and} \quad \lim_{\alpha \to \infty} -\sqrt{\frac{\alpha}{1+\alpha}}=-1.$$
Hence, $\cm{f_\alpha;\zero}$ can take any value in the interval $(-1,0).$

\end{example}

We next illustrate $\nabla f(\xo)=\zero$ with $\cm{f;\xo} = 0$. 

\begin{example}[$\cm{f;\xo}=0$]Consider the function $f:\dom f=\R^2 \to \R:\x=\bbm x_1&x_2\ebm^\top\mapsto -x_2^2$ at the point $\xo=\zero.$  As a vector-valued function, the contour line at $f(\x)=0$ is given by $L(t)=\bbm 1&0\ebm^\top t.$  We find that the set $D^c$ is equal to $D^c=\{ \bbm 1&0\ebm^\top, \bbm -1&0 \ebm^\top \}$ and  $\cm{f;\zero}=0$.
\end{example}

The following example provides a class of functions where $\nabla f(\xo)=\zero$ with $\cm{f;\xo} \in (0,1).$

\begin{example}[$\cm{f;\xo} \in (0,1)$] \label{ex:positiveCosineConstant}
Consider the class of functions 
$$f_\beta:\dom f=\R^2 \to \R:\x \mapsto -\beta x_1^2+ x_2^2,$$ 
where $\beta>0$, at the point $\xo=\zero.$  For any  values of $\beta>0,$ we have $\nabla f_\beta(\zero)=\zero.$
   As a vector-valued function, the contour lines at $f(\x)=0$ have equation  $L_1(t)= \begin{bmatrix} \frac{1}{\sqrt{1+\beta}}&\sqrt{\frac{\beta}{1+\beta}} \end{bmatrix}^\top t,$ and $L_2(t)=\begin{bmatrix} \frac{1}{\sqrt{1+\beta}}&-\sqrt{\frac{\beta}{1+\beta}}\end{bmatrix} t.$ Let $\dd_1=\begin{bmatrix} \frac{1}{\sqrt{1+\beta}}&\sqrt{\frac{\beta}{1+\beta}} \end{bmatrix}^\top$ and $\dd_2= \begin{bmatrix} \frac{1}{\sqrt{1+\beta}}&-\sqrt{\frac{\beta}{1+\beta}}\end{bmatrix}^\top.$
The set $D^c$ can be written as  $D^c=\npspan \{\dd_1,-\dd_2\} \cup \npspan \{-\dd_1,\dd_2\}$.  
We obtain 
$$\cm{f;\zero}=\frac{1}{\sqrt{1+\beta}}$$ and the cosine vectors are $\uu_1=\ee_1$ and $\uu_2=-\ee_1.$
As a function of $\beta,$ $\cm{f;\zero}$ is continuous on $(0,\infty).$ Additionally,
 $$\lim_{\beta \to 0} \frac{1}{\sqrt{1 +\beta}}=1 \quad \text{and} 
 \quad \lim_{\beta \to \infty} \frac{1}{\sqrt{1 +\beta}}=0.$$
This shows  we can pick  $\beta>0$  to create a function where   the cosine measure of $f$ at $\xo=\zero$ can take any value  in $(0,1).$
\end{example}

Finally, we provide an example where $\nabla f(\xo)=\zero$ and $\cm{f;\xo} =1.$

\begin{example}[$\CM(f;\xo)=1$]\label{ex:cm=one}
Consider the function $f:\dom f=\R^2 \to \R:\x \mapsto -x_1^3+x_2^2$ at  the point $\xo=\zero.$ The contour lines  at $f(\x)=0$ are given by the equations  $x_2=\pm x_1^{1.5}.$
The only   descent direction  at  $\xo=\zero$  is $\dd=\ee_1.$ Thus $D^c = \mathbb{S}^2\setminus \{\ee_1\}$ and $\cm{f;\zero}=1$.
\end{example}


In the next section, pseudo-codes are provided showing how the cosine measure of a function may be computed for functions.  The examples above are all sufficiently simple that these codes can be applied (by hand) to confirm the cosine measures.

\section{Computing the cosine measure of a function} \label{sec:algo}
In this section, we show how to compute the cosine measure of a function at a point. The pseudo-code is divided in two cases: $\cl D^c(f;\x)$ is contained in a closed half-space of $\R^n$ and $\cl D^c (f;\x)$ is not contained in a closed half-space of $\R^n$. To make notation tighter, we write $D^c$  and $D$ to represent $D^c(f;\x)$ and $D(f;\x)$, respectively.

\subsection{The case where  $\cl D^c$ is contained in a closed half-space of $\R^n$}

We begin by considering the case where $\cl D^c$ is contained in a closed half-space of $\R^n.$  In this case, the cosine measure value is non-positive.  We assume that the closure of the set  of non-descent directions can be written in the form  
\begin{align}\label{eq:unionpspanginite}
\cl D^c&=\bigcup_{i=1}^k \npspan(\ess_i)
\end{align}
where each $\ess_i$ is a finite set of nonzero vectors.  This structure allows $\cl D^c$ to contain infinite vectors and does not force the cone $\{\vv : \vv = \lambda \dd, \dd \in \cl D^c, \lambda \geq 0\}$ to be convex.  Recall the definition of the polar cone of a set $\Set.$
\begin{definition}
Let $\Set \subseteq \R^n$ be non-empty. 
The \emph{polar cone} of $\Set$ is denoted by $\polar(\Set)$ and defined by
\[
\polar(\Set) \;=\; \{\vv \in \R^n : \vv^\top \dd \le 0 \text{ for all } \dd \in \Set\}.
\]
\end{definition}

Next, we present the pseudo-code for this case, then we show that it returns the exact cosine measure and cosine vector set. 

\begin{center}
\begin{algorithm}[H] 
\caption{Computing $\CM(D^c)$ when $\cl D^c$ is contained in a closed half-space of $\R^n$ \label{alg:cmpspanS}}
Given a set of non-descent directions contained in a closed half-space of $\R^n$ with its closure written in the form 
$$\displaystyle \cl D^c=\bigcup_{i=1}^k\npspan(\ess_i),$$ 
where each $\ess_i$ is a finite set of nonzero vectors  in $\R^n$. 

\textbf{0. Special cases and initialization:}

\quad Define the normalized set $$\displaystyle \ess=\bigcup_{i=1}^k \{ \dd/\|\dd\| : \dd \in \ess_i, i=1, 2, \ldots, k\}.$$


\quad Define $q = \vert \ess \vert.$

\quad If $q=0,$ then  return $\CM(D^c)=-1$ and $\V(D^c)=\emptyset.$ 

\textbf{1. Optimize:} 

\quad Solve the second-order conic program (SOCP)
$\begin{aligned}
    \qquad \quad (1.1)&\quad (t^*, \uu^*) \in \argmin_{t,\;\uu} \{ t : \Vert \uu \Vert \leq 1, \; \uu^\top \dd\leq t \quad \forall \, \dd \in \ess\}
\end{aligned}$

\textbf{2. Compute cosine measure and the cosine vector set:}
return  

$\begin{aligned}
    \qquad \quad (2.1)&\quad \CM(D^c)=t^*.
\end{aligned}$

\quad If $\CM(D^c)<0$, then return

$\begin{aligned}
\qquad \quad (2.2.1)&\quad   \V(D^c)=\{\uu^*\}.
\end{aligned}$

\quad Else ($\CM(D^c)=0)$, 

$\begin{aligned}
\qquad \quad (2.2.2) \quad & \V(D^c)=\polar(\Set) \cap \Sn.
\end{aligned}$
\end{algorithm}
\end{center}

Before proving that Algorithm \ref{alg:cmpspanS} returns the correct output, we make a few remarks.   Note that Algorithm \ref{alg:cmpspanS} requires solving a SOCP. In MATLAB, for example, this can be done be with the function {\tt coneprog}.  Finally, we note that, by  \cite[Theorem 3.2]{Regis2016}, each $\ess_i$ can be reduced to a positive linearly independent subset of $\ess_i$ with maximal cardinality.  In \cite[Section 4]{Regis2016}, Regis refers to this as a {\em frame} of  $\pspan(\ess)$. The vectors of a frame are sometimes called the \emph{generators}. If all generators of a frame for $\pspan(\Set)$ lie in an open half-space, then $\pspan(\Set)$ is said to be \emph{pointed}.  Note that the frame of a pointed polyhedral cone is unique (up to positive scaling).  Reducing each $\ess_i$ to a frame  could reduce the size of the second-order conic program in Step 1 and the (possible) cost of finding the polar cone of $\Set$, resulting in faster completion of the algorithm.

The following theorem provides the main result to prove Algorithm \ref{alg:cmpspanS} returns the correct results.

\begin{theorem}\label{thm:reducetofinite}
Let $\tee=\bigcup_{i=1}^k \npspan(\ess_i)$, where each $\ess_i$ is a finite set 
of nonzero vectors in $\R^n$, and let $\ess=\bigcup_{i=1}^k \ess_i$. If $\tee$ 
is contained in a closed half-space of $\R^n$, then
$$\CM(\tee)=\CM(\ess) \qquad \text{and} \qquad \V(\tee)=\V(\ess).$$
\end{theorem}

\begin{proof}
Since $\tee$ is contained in a closed half-space, $\ess$ is contained in a closed half-space. This means that $\CM(\tee) \leq 0$ and $\CM(\ess) \leq 0.$ Without loss of generality, assume all vectors in $\ess$ are unit vectors.  Since 
$\ess \subseteq \tee$, we have
$\CM(\ess) \leq \CM(\tee)$.

For a 
unit vector $\uu \in \R^n$, write $f_\ess(\uu)=\max_{\dd \in \ess} \uu^\top \dd$ 
and $f_\tee(\uu)=\max_{\dd \in \tee} \uu^\top \dd$.
We claim that $f_\tee(\uu)=f_\ess(\uu)$ whenever $f_\ess(\uu) \leq 0$. Indeed, 
let $\dd \in \tee$. Say $\dd \in \npspan(\ess_i)$ with 
$\ess_i=\{\dd_1, \dd_2, \dots, \dd_q\}$. We may write  $\dd=\sum_{j=1}^q \alpha_j \dd_j$ 
with $\alpha_j \geq 0$. By the triangle inequality, we know 
$\sum_{j=1}^q \alpha_j \geq \Vert \dd \Vert = 1$. We obtain
$$\uu^\top \dd=\sum_{j=1}^q \alpha_j\, \uu^\top \dd_j 
\leq \Big(\sum_{j=1}^q \alpha_j\Big) f_\ess(\uu) \leq f_\ess(\uu),$$
where the last inequality uses $f_\ess(\uu) \leq 0$. Since $\dd \in \tee$ was arbitrary, $f_\ess(\uu)$ is an upper bound for 
$\uu^\top \dd$ over $\tee$. Hence, $f_\tee(\uu) \leq f_\ess(\uu)$. Clearly,  
$f_\ess(\uu) \leq f_\tee(\uu)$ since  $\ess \subseteq \tee.$ This proves the claim.  

For  
any $\uu_\ess \in \V(\ess)$, the previous result gives 
$\CM(\tee) \leq f_\tee(\uu_\ess)=f_\ess(\uu_\ess)=\CM(\ess) \leq 0$. Therefore 
$\CM(\tee)=\CM(\ess)$.

Finally, if $\uu \in \V(\tee)$, then 
$\CM(\ess) \leq f_\ess(\uu) \leq f_\tee(\uu)=\CM(\tee)=\CM(\ess)$, so 
$\uu \in \V(\ess)$. Conversely, if $\uu \in \V(\ess)$, then 
$f_\ess(\uu)=\CM(\ess) \leq 0$,  and we get 
$f_\tee(\uu)=\CM(\ess)=\CM(\tee).$ That is, $\uu \in \V(\tee)$. \qed
\end{proof}

Theorem \ref{thm:reducetofinite} shows that the cosine measure of the (possibly)  infinite set $\bigcup_{i=1}^k\npspan(\ess_i)$ can be found by computing the cosine measure of the finite set $\ess = \bigcup_{i=1}^k \ess_i.$  This allows the application of the techniques from \cite{audet2024cosine}.  We now provide three results on the cosine vector set.

\begin{lemma} \label{lem:oneVector}
Let $\ess$ be a  non-empty set of unit vectors in $\R^n.$  If $\CM(\ess)<0$, then  $\V(\ess)$ contains exactly one vector.
\end{lemma}
\begin{proof}
Since $\ess$ is a non-empty set, $\V(\ess)$ contains at least one vector. We show that it contains exactly one by contradiction. Suppose $\CM(\ess)=\alpha<0$ and $\uu,\vv$ are two distinct unitary vectors contained in $\V(\ess).$ This means  $\uu^\top \dd \leq \alpha$ and $\vv^\top \dd \leq \alpha$  for all $\dd \in \ess.$ 
Consider the unit vector $\frac{\uu+\vv}{\Vert \uu+\vv \Vert}.$ Note that $\Vert \uu+\vv \Vert \neq 0$ since $\Vert \uu+\vv \Vert=0$ implies $\vv=-\uu$ and this would imply  $\vv^\top \dd=-\uu^\top \dd \geq -\alpha>0$. 
Then for any $\dd \in \ess$, we obtain 
$$\left (\frac{\uu+\vv}{\Vert \uu+\vv \Vert} \right )^\top \dd \leq \frac{2\alpha}{\Vert \uu+\vv \Vert} < \alpha,$$
as $\Vert \uu+\vv \Vert < \|\uu\|+\|\vv\| = 2$ and $\alpha <0$.  This is a contradiction to the assumption that $\CM(\ess)=\alpha.$ \qed
\end{proof}

Next, we prove that the  vector $\uu^*$ found while solving the SOCP   in Step (1.1) is a unitary vector  whenever  $\CM(D^c)<0.$

\begin{lemma} \label{lem:uIsUnitary}
Let $\ess$ be a finite set of  unit vectors. Suppose  that the convex relaxation of the cosine measure problem  returns a negative value, i.e.,
\begin{align}\label{eq:convexrelaxationneg}
\min_{\substack{\uu \in \R^n\\\Vert \uu \Vert \leq 1}} \max_{\dd \in \ess} \uu^\top \dd<0.
\end{align}
If $\uu$ is a vector solving \eqref{eq:convexrelaxationneg}, then $\Vert \uu \Vert=1.$
\end{lemma}
\begin{proof}
By way of contradiction, suppose $\uu$ solves \eqref{eq:convexrelaxationneg}. If $\Vert \uu \Vert=0,$ then  $\uu^\top \dd=0$ for all $\dd \in \ess,$ a contradiction.  

Suppose $0<\Vert \uu \Vert < 1.$ Then  $$\left ( \frac{\uu}{\Vert \uu \Vert} \right )^\top \dd<\uu^\top \dd<0$$ for all $\dd \in \ess.$ It follows that $$\max_{\dd \in \ess} \left ( \frac{\uu}{\Vert \uu \Vert}\right )^\top \dd<\max_{\dd \in \ess} \uu^\top \dd,$$ a contradiction to the assumption that $\uu$ solves \eqref{eq:convexrelaxationneg}. \qed
\end{proof}

 \noindent Next, we show that the cosine vector set of $D^c$ is equal to $\polar(D^c) \, \cap \, \Sn$  whenever $\CM(D^c)=0.$

\begin{proposition}\label{prop:cvEqualPolar}
Let $\Set$ be a finite set of unit vectors in $\R^n.$ If $\CM(\Set)=0$, then $\CV(\Set)=\polar(\Set) \, \cap \, \Sn.$
\end{proposition}
\begin{proof}
Let $\uu \in\CV(\Set).$ Then $\max_{\dd \in \Set} \uu^\top \dd=0$ and so
$ \uu^\top \dd\leq0$ for all $\dd \in \Set$. Hence, $\uu \in \polar(\Set) \, \cap \, \Sn.$
Now, suppose $\uu \in \polar(\Set) \, \cap \, \Sn$. This implies
$\max_{\dd \in \Set} \uu^\top \dd\leq0.$ Since $\CM(\Set)=0$, we must have $\max_{\dd \in \Set} \uu^\top \dd=0$ (otherwise, it contradicts the definition of cosine measure).  It follows that $\uu \in \CV(\Set)$.  \qed
\end{proof}

Note that 
\begin{equation} \label{eq:polarEqs}
\polar\left (\bigcup_{i=1}^k \pspan(\Set_i)\right )=\bigcap_{i=1}^k \polar(\pspan(\Set_i))=\bigcap_{i=1}^k \polar(\Set_i)=\polar(\Set)
\end{equation}
where $\Set=\cup_{i=1}^k \Set_i.$  The polar cone is finitely generated (since $\cl(D^c)$ is assumed to have the form described in \eqref{eq:unionpspanginite}).  A frame for the polar cone  can be found  using the MATLAB toolbox MPT3 \cite{MPT32013}.  We conclude this section by showing that Algorithm \ref{alg:cmpspanS} returns the correct value of cosine measure and the complete cosine vector set.

\begin{theorem}
    Given a set of non-descent directions contained in a closed half-space of $\R^n$  written in the form $\displaystyle \cl D^c=\bigcup_{i=1}^k\npspan(\ess_i)$ where each $\ess_i$ is a finite set of nonzero vectors  in $\R^n$, Algorithm \ref{alg:cmpspanS} returns the correct cosine measure value and the complete cosine vector set.
\end{theorem}

\begin{proof}
Lemma 33 of \cite{audet2024cosine} notes that if the cosine measure is nonpositive, then the convex relaxation of the cosine measure problem is exact, i.e., if $\CM(\ess) \leq 0$, then 
    \begin{equation}\label{eq:convexrelaxexact}
        \CM(\ess) = \min_{\substack{\uu \, \in \, \R^n \\\Vert \uu \Vert \leq 1}}\max_{\dd \in \ess} \uu^\top \dd.
    \end{equation}
Step 1 is a standard reformulation of Problem \ref{eq:convexrelaxexact}. Step (2.1) extracts $\CM(D^c),$ given by $t^*$. 

If $\CM(D^c)=0, $  Step (2.2.2) returns the full cosine vector set of $D^c$ by Proposition \ref{prop:cvEqualPolar} and \eqref{eq:polarEqs}.\\ If $\CM(D^c)<0$,  Step (2.2.1) returns the unique cosine vector of $D^c$ by Lemmas \ref{lem:oneVector} and \ref{lem:uIsUnitary}.  \qed
\end{proof}

\subsection{The case where $\cl D^c$ is not contained in a closed half-space of $\R^n$}

Next, we consider the case where $\cl D^c$ is not contained in a closed half-space of $\R^n.$ In this case, the cosine measure is strictly positive.   To find the value of cosine measure and the cosine vector set, we will work from the \emph{spherical complement} of $\cl D^c$, $\inttD =\{\x \in \Sn : \x \notin \cl D^c\}$.  In this case, we assume that the spherical complement of the closure of the non-descent directions can be written in the form  
  \begin{align} \label{eq:formDescentSet}
 \inttD &= \Sn \cap \bigcup_{i=1}^k \intt \pspan (\ess_i), 
 \end{align} 
where $\pspan (\ess_i)\cap \pspan (\ess_j) = \{\zero\}$ for all $i \neq j$.  Intuitively, $\inttD$ is providing a proxy for the interior of $D$ relative to the unit sphere (it is not the interior, nor is it the relative iterator, as both of those are empty).  Notice that $\cl D^c$ not contained in a closed half-space does not imply that $\inttD$ is contained in a closed half-space.  Indeed, the function and point given in Example \ref{ex:threewilldo} has $\cl D^c = \{ \x \in \R^2 : |x_1| \geq |x_2|, x_1^2+x_2^2=1 \}$ and $\inttD = \{ \x \in \R^2 : |x_1| < |x_2|, x_1^2+x_2^2=1 \}$, neither of which are contained in a closed half-space.

If $\cl D^c$ is the entire unit sphere, then $\CM(D^c) = 1$ by Theorem \ref{thm:maintheorem}.  This creates a special case in Algorithm \ref{alg:notContainedSOCP} that we elaborate  in the next lemma.

\begin{lemma}\label{lem:cmfxEqualOne}
Let $f:\dom f \subseteq \R^n \to \R$  and  $\x \in \intt \dom f.$ Then the following are equivalent:
\begin{enumerate}[(i)]
\item $\cl D^c=\Sn$
\item $\inttD=\emptyset,$
\item $\CM(D^c)=1,$ 
\item $\CV(D^c)=\Sn.$ 
\end{enumerate}
\end{lemma}
\begin{proof}
 The result follows immediately from \cite[Prop. 2]{audet2024cosine}. \qed
\end{proof}

Algorithm \ref{alg:notContainedSOCP} proposes a  pseudo-code to compute the cosine measure and the cosine vector set for a function when $\cl D^c$ is not contained in a closed half-space of $\R^n$.

\medskip

\begin{algorithm}[H]
\caption{Computing $\CM(D^c)$ when $D^c$ is not contained in a closed half-space
\label{alg:notContainedSOCP}}
Given a set of non-descent directions not contained in a closed half-space of $\R^n$ with its spherical complement written in the form
$$\inttD=\Sn \cap \bigcup_{i=1}^k \intt \pspan (\ess_i),$$
where each $\ess_i$ is a finite set of nonzero vectors in $\R^n$ and $\pspan (\ess_i) \cap \pspan(\ess_j)=\{\zero\}$ for all $i \neq j$.

\medskip
\textbf{0. Special cases and initialization:}

\quad If $\inttD=\emptyset$, return $\CM(D^c)=1$,
$\V(D^c)=\Sn$.


\quad Else define the index set of nontrivial $\ess_i$, $J =\{i : \intt \pspan(\ess_i)\neq\emptyset\} \subseteq \{1, 2, \ldots, k\}$.

\medskip
\textbf{1. Outer loop:}
For each $i \in J$ do the following.

\quad Compute the set of positive unit normals to the $(n{-}1)$-dim. faces of $\pspan(\ess_i)$ 

$\begin{aligned}
    \qquad \quad (1.1)&\quad  \mathcal{P}_i \;=\;
    \left\{\, \pee \;:\;
    \begin{array}{l}
    \|\pee\|=1, \pee^\top  \dd  \ge \zero \text{ for all } \dd \in \ess_i, \\
    \dim(\spann(\dd : \pee^\top \dd = 0, \dd \in \ess_i)) = n-1 \\
    \end{array}
    \right\}.
\end{aligned}$

\quad Solve the second-order conic program (SOCP)

$\begin{aligned}
    \qquad \quad (1.2)&\quad (t_i^*,\uu_i^*) \in \argmax_{t,\;\uu} \{t: \|\uu\| \;\le\; 1, \; \pee^\top\uu \;\ge\; t
    \quad \forall\,\pee\in\mathcal{P}_i \}
\end{aligned}$

\medskip
\textbf{2. Compute the cosine measure and cosine vector set:} return
$\begin{aligned}
     \qquad \quad(2.1) &\quad\CM(D^c)
    \;=\;
    \min_{i \in J}
    \sqrt{1 - (t_i^*)^{2}},
    \\[4pt]
     \qquad \quad(2.2) &\quad\V(D^c)
    \;=\;
    \bigl\{\,
        \uu_i^*
        \;:\;
        t_i^*
        = \max_{j \in J} t_j^*,
        ~~ i \in J
    \,\bigr\}.
\end{aligned}$
\end{algorithm}

\medskip

In words, Algorithm \ref{alg:notContainedSOCP} considers the pointed polyhedral cones $C_i=\pspan(\Set_i).$ For each $C_i$, it finds the unit vector $\uu_i^*$ which is as far as possible  from the $(n-1)$ dimensional faces of $C_i.$ To do that, the algorithm finds the positive normal vectors for all faces of $C_i$.  Working with the set of positive normals $\Pe_i$ associated to $C_i$, the problem requires to find the unit vector $\uu_i^*$ which is a close as possible (in angle) to the farthest positive normal vectors in $\Pe_i.$ Once it is done for all pointed polyhedral cone $C_i$,  the cosine measure is obtained by taking the minimum value of $\sin(\varphi_i)$ where $0<\varphi_i<\pi/2$ is the angle between $\uu_i^*$ and an active vector $\pe_i \in \Pe_i.$   

Before proving that Algorithm \ref{alg:notContainedSOCP} returns the correct outputs, we make a few more remarks.

First, the assumption that $\pspan (\ess_i) \, \cap \, \pspan(\ess_j)=\{\zero\}$  is necessary  to obtain the exact value of cosine measure.  For example, suppose  
$$\inttD = \{ (x_1, x_2) : 0<x_1, -x_1 < x_2, \|\x\|=1\}.$$
The desired description would use $\{(\ee_1-\ee_2)/\sqrt{2}, \ee_2\}$, while an undesirable description would use $\{(\ee_1-\ee_2)/\sqrt{2}, (\ee_1+\ee_2)/\sqrt{2}\}$ and  $\{\ee_1, \ee_2\}$  (see Figure \ref{fig:OverlappingCones}).  The correct representation gives $$\CM(D^c)=\cos(3\pi/8) \approx 0.3827.$$  However, the ``overlapping'' description gives $\CM(D^c)=\cos(\pi/4)\approx 0.7071.$  If $\pspan (\ess_i) \, \cap \, \pspan(\ess_j) \neq \{\zero\}$ for some  $i \neq j,$ then the value of cosine measure may be overestimated.

\begin{figure}[H]
    \centering
    \includegraphics[width=\linewidth]{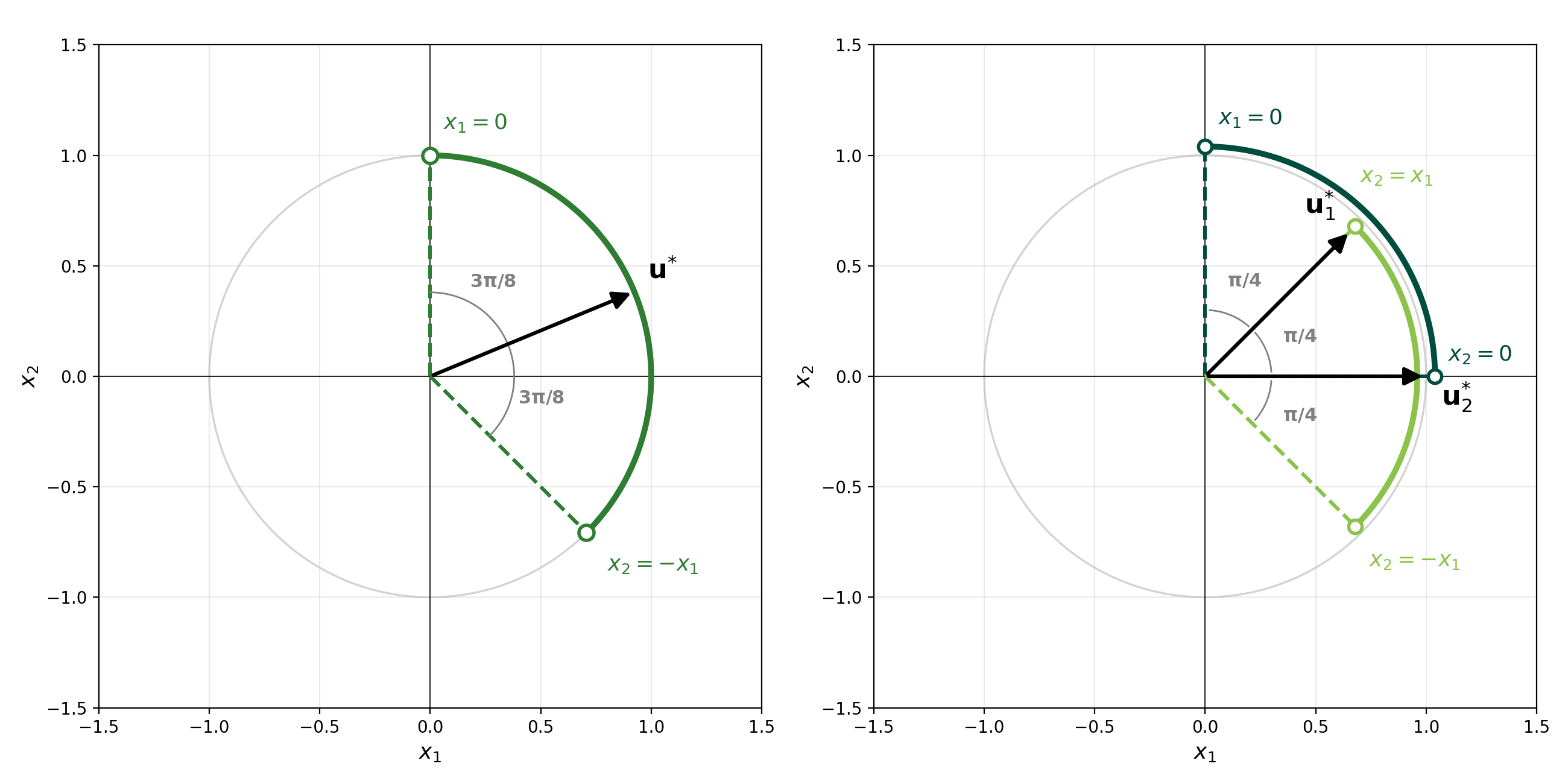}
    \caption{On the left,  the set $\{ (x_1, x_2) : 0<x_1, -x_1 < x_2, \|\x\|=1\}$ is represented by one spherical cap. The cosine vector $\uu^*$ is found using Algorithm \ref{alg:notContainedSOCP}.  On the right, the set is decomposed into two overlapping spherical caps.  A vector $\uu^*$  is computed for each spherical cap. This yields an incorrect (overestimated) value of the cosine measure.}
    \label{fig:OverlappingCones}
\end{figure}

Examining the special cases (Step 0), we notice that if Algorithm \ref{alg:notContainedSOCP} is applied to a single variable function (i.e., $n=1$), then ``$\cl D^c$ is not in a closed half-space'' implies that $\cl D^c = \mathbb{S}^1$, so the special case  $\inttD=\emptyset$ is applied.   Also note, in Step 0, that if $J =\emptyset$, then the special case $\inttD=\emptyset$ would have been triggered. Therefore, if the algorithm continues to Step 1, then we must have $n \geq 2$ and $J \neq \emptyset$.

Like Algorithm \ref{alg:cmpspanS}, if desired, each $\ess_i$ could be reduced to its unique frame at the beginning of Step 1.  This can simplify the process to identify all boundary faces of a pointed polyhedral cone. 

In Step 1, the algorithm defines $\mathcal{P}_i$ to be the set of positive unit normal vectors to the $n-1$ dimensional faces of $C_i=\pspan(\ess_i)$.  To justify this, first notice that $\dim \pspan(\ess_i) = n$, as otherwise $\intt \pspan(\ess_i) = \emptyset$ and the set would have been ignored in Step 1.  As $\dim \pspan(\ess_i) = n$, the set $\mathcal{P}_i$ can be created by taking subsets of $\ess_i$ of size $n-1$ using the (generalized) cross product to create $\tilde{\pee}$ then checking if either $\pee= \tilde{\pee}/\|\tilde{\pee}\|$ or $\pee=- \tilde{\pee}/\|\tilde{\pee}\|$ satisfies $\pee^\top \dd \geq \zero$. It is possible to identify the same face of $\pspan(\ess_i)$ more than once by considering sets of $n-1$ linearly independent vectors contained in $\ess_i$. For example, consider the set $\ess_i=\{\pm \ee_1, \ee_2, \ee_3\}\subset \R^3$.  The linearly independent sets $\{\ee_1, \ee_2 \}$ and $\{-\ee_1, \ee_2\}$ both find the unit normal vector $\pee=\ee_3$.  Similarly, the linearly independent sets $\{\ee_1, \ee_3 \}$ and $\{-\ee_1, \ee_3\}$ both find the unit normal vector $\pee=\ee_2$.  As such, when constructing $\mathcal{P}_i$ it would be wise to implement a strategy to detect and remove repeated vectors.  Notice that this process is identifying the facets of a polyhedral cone, which can grow exponential in dimension.

Step 1 of Algorithm \ref{alg:notContainedSOCP} then requires solving a second-order conic program for each $C_i$.  Note that the  constraint $\|\uu\|\le 1$ in Step (1.2) relaxes the sphere constraint $\|\uu\|=1$.  Since $t_i^*>0$ for every pointed full-dimensional cone, scaling $\uu$ toward the unit sphere strictly improves the objective function $t$. Hence, every optimal solution  satisfies $\|\uu_i^*\|=1$. Furthermore, note that  $\uu_i^*$ must be in $\intt \pspan(\Set_i)$ since $\pe^\top \uu_i^*>0$ for all $\pe \in \Pe_i$ (this is used in Theorem \ref{thm:Alg2works}). 

We now focus on proving that Algorithm \ref{alg:notContainedSOCP} returns the correct output.

\begin{lemma}\label{lem:boundcm}
Suppose $D^c$ is a set of non-descent directions not contained in a closed half-space of $\R^n$   with its spherical complement written in the form
    $$\inttD=\Sn \cap \bigcup_{i=1}^k \intt \pspan (\ess_i),$$
where each $\ess_i$ is a finite set of nonzero vectors in $\R^n$ and $\pspan (\ess_i) \, \cap \, \pspan(\ess_j)=\{\zero\}$ for $i \neq j$.  For each $i \in \{1, 2, \ldots, k\}$ define $\ess^\circ_i = \Sn \, \cap \, \intt \pspan (\ess_i)$.  
Then, for all $i$, 
    $$0<\cm{(\ess^\circ_i)^c}\leq1,$$
with $\cm{(\ess^\circ_i)^c}=1$ if and only if $\ess^\circ_i = \emptyset$.
\end{lemma}
\begin{proof} Since $D^c$ is not contained in a closed half-space, it follows from Theorem \ref{thm:maintheorem} that $0<\cm{D^c} \leq \cm{(\ess^\circ_i)^c}$ for all $i$.  The second statement follows from Lemma \ref{lem:cmfxEqualOne}.\qed
\end{proof}

Next, we show that the cosine measure of $D^c$ is obtained by considering  the appropriate single set $\ess_i$.  This greatly simplifies the remaining proofs.

\begin{lemma}\label{lem:onewilldo}
    Suppose $D^c$ is a set of non-descent directions not contained in a closed half-space of $\R^n$ such that $\cl D^c \neq \Sn$ and with its spherical complement written in the form
$$\inttD=\Sn \cap \bigcup_{i=1}^k \intt \pspan (\ess_i),$$
where each $\ess_i$ is a finite set of nonzero vectors in $\R^n$ and $\pspan (\ess_i) \,\cap \, \pspan(\ess_j)=\{\zero\}$ for $i \neq j$.  For each $i \in \{1, 2, \ldots, k\}$ define $\ess^\circ_i = \Sn \, \cap \, \intt \pspan (\ess_i)$.
Then there exists $i \in \{1, 2, \ldots, k\}$ such that
    $$\cm{D^c} = \cm{(\ess^\circ_i)^c}.$$
\end{lemma}

\begin{proof}
Note that $\ess_i^\circ$ is non-empty for at least one $i$ since $\cl(D^c) \neq \Sn$. Without loss of generality, assume that $\ess_i^\circ \neq \emptyset$ for all $i \in \{1, 2, \dots, k\}.$  Since $ D^c \subseteq (\ess^\circ_i)^c$ for every $i$, we get
$\cm{D^c} \leq \cm{(\ess^\circ_i)^c}$ for all $i.$

For the reverse inequality, let $\uu^* \in \CV(D^c)$ and define 
$\gamma = \cm{D^c}<1.$ Consider the open spherical cap
$$\mathcal{C} = \{\dd \in \Sn : \dd^\top \uu^* > \gamma\}.$$
Since $\gamma < 1$, we must have  
$\uu^* \in \mathcal{C}$, so $\mathcal{C} \neq \emptyset$. By definition of 
$\cm{D^c}$, every $\dd \in D^c$ satisfies $\dd^\top \uu^* \leq \gamma$. Hence, 
$\mathcal{C} \cap  D^c = \emptyset$  and it follows
$\mathcal{C} \subseteq \bigcup_{i=1}^k \ess^\circ_i$.  Since 
$\pspan(\ess_i) \cap \pspan(\ess_j) = \{\zero\}$ for $i \neq j,$  and   $\mathcal{C}$ is connected, there exists a unique index $i^*$ 
such that $\mathcal{C} \subseteq \ess^\circ_{i^*}$.  Therefore,  every 
$\dd \in (\ess^\circ_{i^*})^c$ satisfies $\dd^\top \uu^* \leq \gamma$. This implies
$$\cm{(\ess^\circ_{i^*})^c} \;\leq\; \max_{\dd \in (\ess^\circ_{i^*})^c} 
\dd^\top \uu^* \;\leq\; \gamma \;=\; \cm{D^c}.$$
\qed
\end{proof}

Lemma \ref{lem:onewilldo} allows us to simplify the remaining proofs, as it allows us to  consider  each $\pspan(\ess_i)$ one at a time.   That is, it suffices to show that $(t^*_i, \uu^*_i)$, from Step 1 of Algorithm \ref{alg:notContainedSOCP}, provides the cosine measure and a cosine vector for $(\ess^\circ_i)^c$.  In Step 2, we simply take the minimum of $\sqrt{1-(t^*_i)^2}$ over $i$ to obtain the cosine measure of $D^c$.

\begin{theorem}\label{thm:Alg2works}
Suppose $D^c$ is a set of non-descent directions not contained in a closed half-space of $\R^n$ with its spherical complement written in the form
    $$\inttD=\Sn \cap \bigcup_{i=1}^k \intt \pspan (\ess_i),$$
where each $\ess_i$ is a finite set of nonzero vectors in $\R^n$ and $\pspan (\ess_i) \, \cap \, \pspan(\ess_j)=\{\zero\}$ for $i \neq j$. 

For each $i \in \{1, 2, \ldots, k\}$ define $\ess^\circ_i = \Sn \, \cap \, \intt \pspan (\ess_i)$.  Let $\pp_i \subset \R^n$ be the finite set containing the unit normal vector $\pee_{j_i}$ for each boundary hyperplane $H_{j_i}$ of $\pspan(\ess_i)$ such that $\pee_{j_i}^\top \x\geq 0$ for all $\x \in \pspan(\ess_i)$ and for all $j_i \in \{1, 2, \dots, m_i\}, i \in \{1, 2, \dots, k\}.$ If $\ess_i^\circ \neq \emptyset$, then  
\begin{align*}
\CM((\ess^\circ_i)^c)&=\max_{\pee \in \pp_i} \sqrt{1-(\pee^\top \uu_i^*)^2},
\end{align*}
where $\uu_i^*$ is the unique unit vector solving the SOCP in Step (1.2) of Algorithm \ref{alg:notContainedSOCP}.  Consequently, Algorithm \ref{alg:notContainedSOCP} returns the cosine measure and the cosine vector set for $D^c$.
\end{theorem}

\begin{proof} Consider $i$ such that $\ess_i^\circ \neq \emptyset$.  Note that this implies $\cm{(\ess_i^\circ)^c} < 1$.   Define the normalized frontier of $\pspan(\Set_i)$ via
    $$\bpspan = \left(\cl \pspan(\Set_i) \setminus \intt \pspan(\Set_i)\right) \cap \Sn.$$
Note that
\begin{align}
    \CM((\ess^\circ_i)^c)&=\min_{\substack{\uu \in \R^n\\ \Vert \uu \Vert=1}} \sup_{\dd \in (\ess^\circ_i)^c} \dd^\top \uu \notag \\
    &= \min_{\substack{\uu \in \R^n\\ \Vert \uu \Vert=1}} \max_{\dd \in \cl((\ess^\circ_i)^c)} \dd^\top \uu \notag \\
    &=\min_{\uu \in \ess^\circ_i} \max_{\dd \in \cl((\ess^\circ_i)^c)}  \dd^\top \uu &&(\text{since $\CM(D^c)<1$, so $\uu$ cannot be in $(\ess^\circ_i)^c$}) \notag \\
    &=\min_{\uu \in \ess^\circ_i}  \, \max_{\substack{\dd \in \cl \pspan(\Set_i) \setminus \intt \pspan(\Set_i)\\ \|\dd\| = 1}} \dd^\top \uu \notag \\
    &=\min_{\uu \in \ess^\circ_i}  \, \max_{\dd \in \bpspan} \dd^\top \uu, \label{eq:uInInterior}
\end{align}
since the maximum is attained by a unit vector on the boundary of $\pspan(\Set_i)$.  

Let $\uu_i^* \in \ess_i^\circ$ be a solution to \eqref{eq:uInInterior}   and let $H_{1_i}, \dots, H_{m_i}$ be the boundary hyperplanes of $\pspan(\ess_i)$.  Denote the projection of the vector $\uu_i^*$ onto the hyperplane $H_{j_i}$ by $\Proj_{H_{j_i}} \uu_i^*$. Also, denote a boundary hyperplane  containing a unit vector $\dd \in \R^n$ by $H_\dd.$ We have
\begin{align*}
0<\max_{\dd \in \bpspan} \dd^\top \uu_i^* &=\max_{\dd \in \bpspan} \dd^\top (\Proj_{H_\dd} \uu_i^*)\\
&=\max_{j_i \in \{ 1, 2, \dots, m_i\}} \left (\frac{\Proj_{H_{j_i}} \uu_i^*}{\Vert \Proj_{H_{j_i}} \uu_i^* \Vert} \right )^\top \Proj_{H_{j_i}} \uu_i^* \\
&=\max_{j_i \in \{1, 2, \dots, m_i\}} \Vert \Proj_{H_{j_i}} \uu_i^* \Vert.
\end{align*}
Note that $\Proj_{H_{j_i}} \uu_i^*$ need not lie in $\pspan(\ess_i)$. However,
 the maximum of $\Vert \Proj_{H_{j_i}} \uu_i^* \Vert$ over $j_i$ 
is attained at a hyperplane realizing the minimal distance from 
$\uu_i^*$, for which $\Proj_{H_{j_i}} \uu_i^* \in  \pspan(\ess_i).$ Next, note that

\begin{align*}
\max_{j_i \in \{1, 2, \dots, m_i\}} \Vert \Proj_{H_{j_i}} \uu_i^* \Vert&=\max_{j_i \in \{1, 2, \dots, m_i\}}\sin(\varphi_{\pee_{j_i}})
\end{align*}
where $0<\varphi_{\pee_{j_i}}<\pi/2$ is the angle between the positive normal vector $\pee_{j_i}$ associated to the hyperplane $H_{j_i}$ and the vector $\uu_i^*$. Since $(\uu_i^*)^\top \pee_{j_i}=\cos(\varphi_{\pe_{j_i}})$ and $\sin(\varphi_{\pe_{j_i}})=\sqrt{1-\cos^2(\varphi_{\pe_{j_i}})},$ we obtain 
\begin{align*}
\cm{(\ess^\circ_i)^c}&=\max_{\pe \in \Pe_i} \sqrt{1-(\pe^\top \uu_i^*)^2}.
\end{align*}
Now, since $\Pe_i$ is a non-empty compact set, and $\phi(t)=\sqrt{1-t^2}$ is a strictly decreasing  continuous function on $0<t<1$,  we have $$\min_{\uu \in \ess_i^\circ} \max_{\pe \in \Pe_i} \sqrt{1-(\pe^\top \uu)^2}=\sqrt{1-\left ( \max_{\uu \in \ess_i^\circ} \min_{\pe \in \Pe_i} \pe^\top \uu\right )^2}.$$ Moreover, the two problems have identical solution sets $$\argmin_{\uu \in \ess_i^\circ} \max_{\pe \in \Pe_i} \sqrt{1-(\pe^\top \uu)^2}=\argmax_{\uu \in \ess_i^\circ} \min_{\pe \in \Pe_i} \pe^\top \uu.$$ 
Next, we prove  that  
\begin{equation} \label{eq:siToRn}
\max_{\uu \in \ess_i^\circ} \min_{\pe \in \Pe_i} \pe^\top \uu=\max_{\substack{\uu \in \R^n \\ \Vert \uu \Vert=1}} \min_{\pe \in \Pe_i} \pe^\top \uu.
\end{equation}
Indeed,  since $\ess_i^\circ$ is  a subset of $\Sn$, we get $$\max_{\uu \in \ess_i^\circ} \min_{\pe \in \Pe_i} \pe^\top \uu \leq \max_{\substack{\uu \in \R^n \\ \Vert \uu \Vert=1}} \min_{\pe \in \Pe_i} \pe^\top \uu.$$ Let  $f(\uu)=\min_{\pe \in \Pe_i} \pe^\top \uu$ and $\uu^*$ be a  unit vector achieving the maximum of the right hand side of \eqref{eq:siToRn}.  Then for any  $\tilde{\uu} \in \ess_i^\circ,$ we have  $f(\uu^*) \geq f(\tilde{\uu})>0.$  Note that $f(\uu^*)>0$ implies $\uu^*$ is in $\ess_i^\circ.$ We get  $$\max_{\tilde{\uu} \in \ess_i^\circ} f(\tilde{\uu}) \geq f(\uu^*)=\max_{\substack{\uu \in \R^n \\\Vert \uu \Vert=1}} f(\uu).$$  Besides,  since the optimal value of $f(\uu)$ is positive, every maximizer over $\Sn$  satisfies $\pe^\top \uu^*>0$. Hence, both sides of \eqref{eq:siToRn} have the same solution set.

Finally, using similar arguments as in Lemmas \ref{lem:oneVector}, \ref{lem:uIsUnitary} and  Lemma 33 in \cite{audet2024cosine}, we obtain that the convex relaxation of the problem  
\begin{equation}\label{eq:nonConvexProb}
\max_{\substack{\uu \in \R^n \\\Vert \uu \Vert=1}} \min_{\pe \in \Pe_i} \pe^\top \uu
\end{equation}
may be considered:   
\begin{equation}\label{eq:ConvexProgram}
\max_{\substack{\uu \in \R^n \\ \Vert \uu \Vert \leq 1}} \min_{\pe \in \Pe_i} \pe^\top \uu.
\end{equation}
The convex relaxation returns the unique unit vector solving the nonconvex program  \eqref{eq:nonConvexProb}. Step (1.2) of Algorithm \ref{alg:notContainedSOCP} is a reformulation of the convex program \eqref{eq:ConvexProgram}. 
In the case where $D^\circ$ consists of several polyhedral cones $\pspan(\ess_i)$, the correctness of Step 2 follows from Lemma~\ref{lem:onewilldo}. \qed
\end{proof}

\begin{remark}
Given a full-dimensional pointed polyhedral cone $C=\pspan(\Set)\subsetneq \R^n,$ we could define the \emph{sine measure of $C$} (denoted by $\sm(C)$) by   
\begin{equation} \label{eq:sineProb}
\sm(C)=\min_{\substack{\uu \in C\\ \Vert \uu \Vert=1}} \max_{\pee \in \pp} \sqrt{1-(\pee^\top \uu)^2}=\min_{\substack{\uu \in C\\ \Vert \uu \Vert=1}} \max_{\pee \in \pp} \sin(\varphi)=\sqrt{1-\left ( \max_{\substack{\uu \in \R^n \\ \Vert \uu \Vert \leq 1}} \min_{\pe \in \Pe} \pe^\top \uu\right )^2}
\end{equation}
where $0<\varphi<\pi/2$ is the angle between $\uu$ and $\pe$,  and where $\Pe$ is the finite set of positive unit normal vectors for the boundary hyperplanes of $C$.  That is, find the unit vector in $C$ that is as far as possible from its nearest boundary hyperplane. Equivalently, in terms of $\Pe$, find the unit vector which is as close as possible (in angle) to the farthest normal in $\Pe.$
A unit vector $\uu$ solving the sine measure problem is called the \emph{sine vector of $C.$} It is worth mentioning that  the sine vector of $C$ is equivalent to the \emph{incenter} of a \emph{proper cone} $K$ denoted by $\pi_{inc(K)}$ and introduced in \cite{HenrionSeeger2010,Henrion2010}.

Under this terminology, Step (1.2) is the second-order conic program to solve the sine measure problem for each $\pspan(\Set_i).$ From Theorem \ref{thm:Alg2works} , each $\uu_i^*$ found while solving the SOCP is a unit vector  in $\intt \pspan(\Set_i).$  However, it should be noted that even though the problem is an SOCP, the number of constraints for this problem grows with the number of boundary hyperplanes to $C$. Indeed, determining the constraints for the SOCP is equivalent to identifying the faces of a pointed polyhedral cone from its extreme rays (which can grow exponentially).
\end{remark}


\section{Computing the cosine measure of  the finite-max function} \label{sec:cmNonSmooth}

In this section, we consider a nonsmooth function of the form 
\begin{equation}\label{eq:maxfunction}
F=\max\{f_1, f_2, \dots, f_\ell\}, \quad \text{where} \quad f_i \in \mathcal{C}^1 \quad \text{for all} \; i \in \{1, 2, \dots, \ell\}.
\end{equation}
We will see that computing the cosine measure of $F$ at $\x$ relies on the active gradients. Indeed, the set $\Pe$ defined in Step (1.1) of Algorithm \ref{alg:notContainedSOCP} is exactly the negative of the set of normalized active gradients at the point $\x$.  A particular type of finite-max function, the $\ell_1$ norm is considered to conclude this section.  We begin by recalling the active  index set at $\x$. 
\begin{definition}[Active set of indices]
 Let $F : \mathbb{R}^n \to \mathbb{R}$ be defined by
\begin{equation*}
    F(\x)=\max_{i\,\in\,\{1, 2, \dots,\ell\}} f_i(\x),
\end{equation*}
where $f_i:\R^n  \to \R$ is in $\mathcal{C}^1.$
The \emph{active index set} of $F$ at $\x$ is denoted by $I(\x)$ and defined by 
\begin{equation*}
\label{eq:active_set}
    I(\x)=
    \bigl\{
        i \in \{1, 2,\dots, \ell\}
        \;:\;
        f_i(\x) = F(\x)
    \bigr\}.
\end{equation*}
\end{definition}
Note that $I(\x)$ is non-empty for any $\x$.

The set of non-descent directions of $F$ at $\x$ is given by 
    $$D^c=\bigcup_{i \in I(\x)}\{\dd \; : \; \nabla f_i(\x)^\top \dd \geq 0, \Vert \dd \Vert=1 \}.$$ 
The spherical complement of $\cl D^c$ in $\Sn$ is equal to  
    \begin{equation}\label{eq:descentDirMaxFunction}
    D^o=\left \{ \bigcap_{i \in I(\x)}\{ \dd \; : \; \nabla f_i (\x)^\top \dd<0, \,  \Vert \dd \Vert=1 \} \right \} \cap \Sn.
    \end{equation}
    
When $\vert I(\x) \vert=1$ and the gradient of the active function is not the zero vector, by Theorem \ref{thm:cmdiffpoint}, we obtain that $\CM(D^c)=0$ and $\V(D^c)=-\nabla f_{i^*} (\x)/\Vert \nabla f_{i^*}(\x) \Vert$ where $I(\x) = \{i^*\}$.

If one of the active gradient vectors is equal to the zero vector, then  the set of non-descent directions of  $F$ at $\x$ can take any form. In particualar, examples may be created (as we did in Section \ref{sec:exampleGradEqualZero}) showing that in this case the cosine measure of $F$ at $\x$ may take any value in [-1,1]. For this reason,  we assume that the set of active gradient vectors does not contain the zero vector in the  following algorithm. 

\medskip

\begin{algorithm}[H]
\caption{Computing $\CM(D^c)$ of a max function
\label{alg:maxFunction}}
Given  a function of the form 
$F=\max\{f_1, f_2, \dots, f_\ell\},$ where $f_i \in \mathcal{C}^1$ for all  $i \in \{1, 2, \dots, \ell\}$ and a point $\x \in \R^n$.

\medskip
\textbf{0. Special cases and initialization:} 

\quad Define the set of active indices  $I(\x)= \{ i \in \{1, 2,\dots, \ell\} \;:\; f_i(\x) = F(\x)\}.$ 

\quad If $\nabla f_i (\x) = \zero$ for any $i \in I(x)$, stop (cosine measure cannot be computed).

\quad If $\vert I(x)\vert=1,$ return $\CM(D^c)=0$,
$\V(D^c)=-\nabla f_{i^*} (\x)/\Vert \nabla f_{i^*} (\x) \Vert,$ where $\{i^*\} = I(\x)$.

\medskip
\textbf{1. SOCP set up and solution:}

\quad Define the set of active negative gradients of $F$ at $\x$:

$\begin{aligned}
    \qquad \quad (1.1)&\quad
    \mathcal{P}=\left \{-\frac{\nabla f_i(\x)}{\Vert \nabla f_i(\x) \Vert} \; : \; i \in I(\x)\right \}
\end{aligned}$

\quad Solve the second-order conic program (SOCP)

$\begin{aligned}
    \qquad \quad (1.2)&\quad (t^*,\uu^*) \in \argmax_{t,\;\uu} \{t:\|\uu\| \;\le\; 1, \pee^\top\uu \;\ge\; t
    \quad \forall\,\pee\in\mathcal{P} \; \}
\end{aligned}$

\medskip
\textbf{2. Compute the cosine measure and cosine vector set:} return
$\begin{aligned}
     \qquad \quad(2.1) &\quad\CM(D^c)
    \;=\;
    \sqrt{1 - (t^*)^{2}},
    \\[4pt]
     \qquad \quad(2.2) &\quad\V(D^c)
    \;=\;
\begin{cases}
    \{\uu^*\} & \text{if } t^* > 0, \\
    \Sn       & \text{if } t^* = 0.
\end{cases}
\end{aligned}$
\end{algorithm}

\smallskip 

Notice that Algorithm \ref{alg:maxFunction} is Algorithm \ref{alg:notContainedSOCP} with a few simplifications.  The set of positive normals is now given by the set of  normalized active negative gradients at $\x$. The outer loop is deleted since  the set of descent directions of $F$ at $\x$ can be represented by a single polyhedral cone $C$ (intersecting with $\Sn$) as described in \eqref{eq:descentDirMaxFunction}. Note that it is possible for $C$ to be empty.  For example, consider $f_1:\R^2 \to \R: \x \mapsto x_1, f_2:\R^2 \to \R: \x \mapsto -x_1$ at the point $\x=\zero.$ Then $F=\vert x_1 \vert$ and we have $\nabla f_1(\zero)=\ee_1$ and $\nabla f_2(\zero)=-\ee_1$ In this case, the SOCP returns $t=0$ and we get $\CM(D^c)=1, \V(D^c)=\mathbb{S}^2.$  Thus we have the following corollary.

\begin{corollary}
Given  a function of the form 
$F=\max\{f_1, f_2, \dots, f_\ell\},$ where $f_i \in \mathcal{C}^1$ for all  $i \in \{1, 2, \dots, \ell\}$ and a point $\x \in \R^n,$ Algorithm \ref{alg:maxFunction} returns the correct cosine measure and cosine vector set.
\end{corollary}

We end with two examples demonstrating Algorithm \ref{alg:maxFunction}.

\begin{example}
\label{ex:cm_R2}
Consider two functions with domain $\R^2$  defined by
\[
    f_1(\x) = x_1+x_2, \qquad f_2(\x) = -x_1+x_2.
\]
Let $F(\x)=\max\{x_1+x_2,\,-x_1+x_2\}=|x_1|+x_2$, and consider the point $\x_0=\zero$.  From the contour plot (see Figure \ref{fig:exampleeasy}), it is easy to confirm that $\CM(F;\zero) = 1/\sqrt{2}$ and $\CV(F;\zero) = \{[0 ~~ -1]^\top\}$.

\begin{figure}[H]
    \centering
    \includegraphics[width=0.95\linewidth]{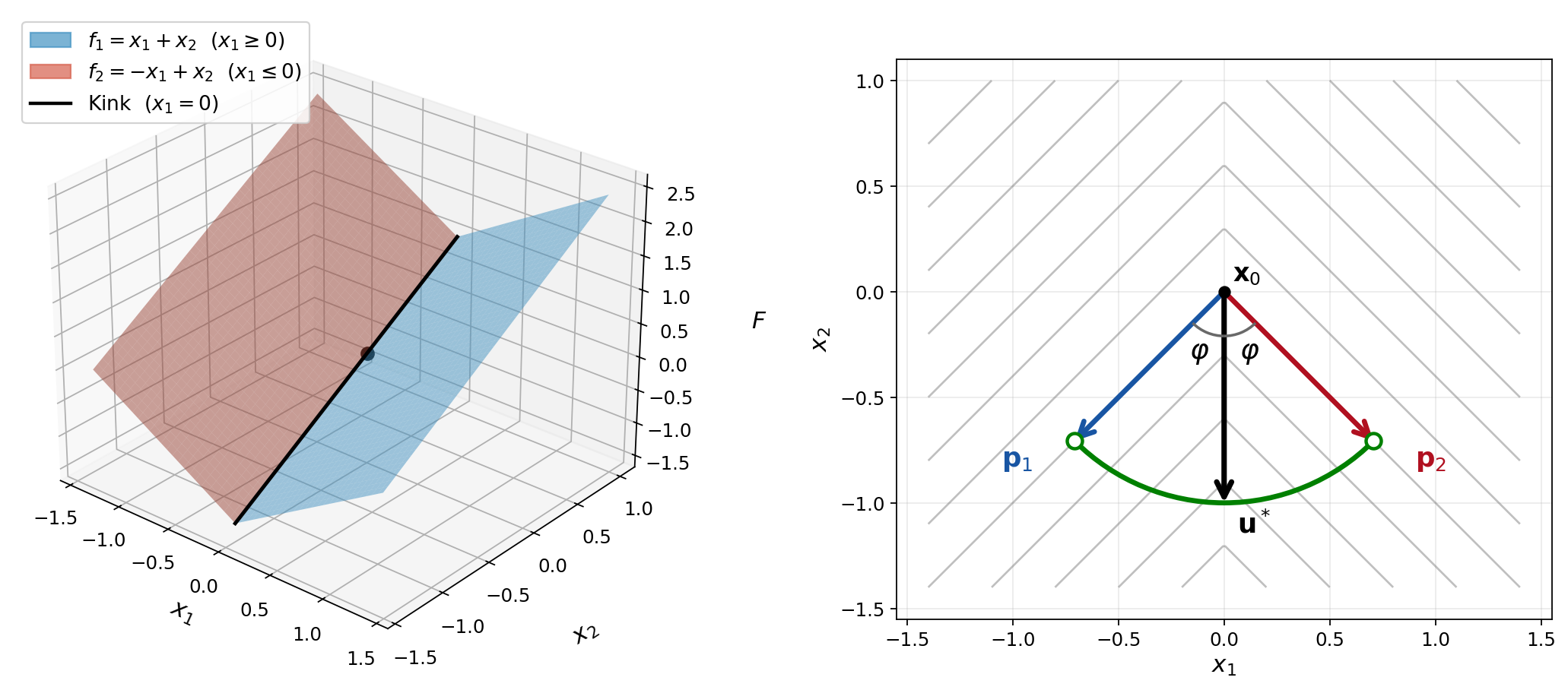}
    \caption{Graph and the descent directions of $F=\max\{x_1+x_2,-x_1+x_2\}$ at $\xo=\zero$ (represented by the green spherical cap).}
    \label{fig:exampleeasy}
\end{figure}

The active set of indices at $\x_0=\zero$ is $I(\x_0)=\{1,2\}$, so the normalized active negative gradient set is 
\[
    \pp=\left\{
        \pee_1=\frac{1}{\sqrt{2}}\bbm-1\\-1\ebm,
        \quad
        \pee_2=\frac{1}{\sqrt{2}}\bbm \ph 1\\-1\ebm
    \right\}.
\]
Solving the SOCP,
$$\begin{array}{rcl}\displaystyle
\max_{t,\;\uu} \left\{t:
        \begin{array}{rl} 
        \|\uu\| & \le 1 \\
        \pee^\top\uu &\ge t \quad \forall\,\pee\in\mathcal{P} \; 
        \end{array}
        \right\}
    &=&\displaystyle
    \max_{t,\;\uu} \left\{t:
        \begin{array}{rl}
            \|\uu\| & \leq 1 \\
            -u_1/\sqrt{2} - u_2/\sqrt{2} &\geq t \\
            u_1/\sqrt{2} - u_2\sqrt{2} &\geq t
        \end{array}     \right\}\\
    &=&\displaystyle
    \max_{t,\;\uu} \left\{t:
        \begin{array}{rl}
            \|\uu\| & \leq 1 \\
             - u_2 &\geq \sqrt{2} t \\
        \end{array}     \right\},\\
    \end{array}
    $$
we obtain
$$\CM(F;\zero)=\sqrt{1-(1/\sqrt{2})^2}=1/\sqrt{2} \quad \mbox{and} \quad \uu^*=\bbm \ph 0\\-1\ebm =\CV(F;\zero).$$
\end{example}

In Example \ref{ex:cm_R2}, the cosine vector $\uu^*$ is also the direction of steepest descent.  In general, this is not necessarily true.  For example, considering $F(\x)=\max\{x_1,2x_2\}$ at the point $\xo=\zero$, we find $\uu^*=-\frac{1}{\sqrt{2}} \bbm1&1\ebm^\top$ (and $\CM(D^c)=1/\sqrt{2}$) but the direction of steepest descent, denoted by $\uu_{sd}$, is $-\frac{1}{\sqrt{5}} \bbm 2&1 \ebm^\top$.  Figure \ref{fig:UsdNotEqualUStar} illustrates both vectors on the contour plot of $F$.
\begin{figure}[H]
    \centering
    \includegraphics[width=0.4\linewidth]{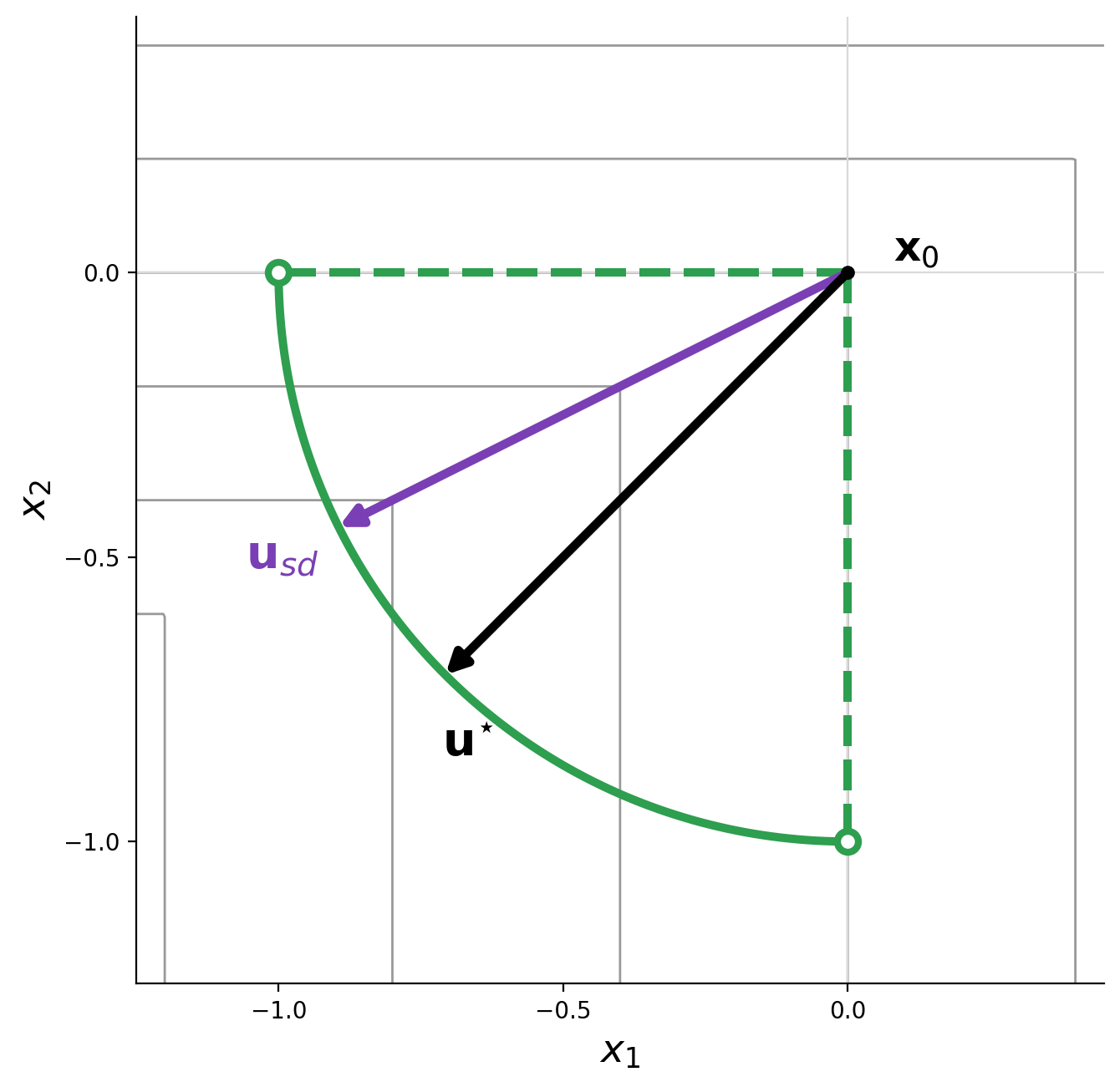}
    \caption{The cosine vector $\uu^*$ of $F$ at $\xo=\zero$ is not equal to the steepest descent direction $\uu_{sd}$}
    \label{fig:UsdNotEqualUStar}
\end{figure}

Our second example considers the common regularization function $\ell_1$.  Recall, in $\R^n$,
    $$\|\x\|_1=\sum_{i=1}^n |x_i| = \max_{\ssig \in \{-1,1\}^n } \ssig^\top \x, $$
where the notation $\{-1,1\}^n$ means that each entry of the n-dimensional vector $\ssig$ is either $-1$ or 1.  Therefore, $\|\x\|_1$ is the maximum of $2^n$ linear functions. Each linear  function has the form
$f_{\ssig}(\x)=\ssig^\top\x\in\mathcal{C}^1$
and gradient equal to $\nabla f_{\ssig}=\ssig \neq \zero_n$.  Algorithm \ref{alg:maxFunction} provides us the tools to compute the cosine measure of this function as a simple formula of the number of zeros in the point $\x_0$. 

\begin{theorem}[Cosine Measure of the $\ell_1$ norm at $\xo$]
\label{thm:cosineMeasureEllOne}
Let $F(\x)=\|\x\|_1$ and let $\xo\in\R^n, n\geq 2$.  Let $m$ be the number of $0$s in $\x_0$.  Then 
    $$\CM(F;\x_0) = \sqrt{\frac{m}{n}} 
    \quad \mbox{and} \quad
    \CV(F;\x_0) =
    \left\{
        \begin{array}{rl}
            \Sn & \mbox{if } m=n \\
            \frac{-\bfs}{\sqrt{n-m}}& \mbox{if } m<n,
        \end{array}
    \right\}
    $$
where $\bfs \in\{-1,0,1\}^n$ is defined by
\begin{equation}
\label{eq:signvec}
  s_j = \sgn\bigl((\xo)_j\bigr)
  =
  \begin{cases}
    +1 & (\xo)_j > 0,\\
    \phantom{+}0 & (\xo)_j = 0,\\
    -1 & (\xo)_j < 0,
  \end{cases}
  \qquad j\in \{1, 2,\ldots,n\}.
\end{equation}
\end{theorem}

\begin{proof}If $m=n$, then $\x_0=\zero$ and the result follows from Theorem \ref{thm:maintheorem}.

Suppose $m<n$ and define the nonzero and zero index sets respectively by
\begin{equation*}
  N=\bigl\{j:(\xo)_j\neq 0\bigr\},\qquad
  Z=\bigl\{j:(\xo)_j=0\bigr\},\qquad
  |N|=n-m,\quad|Z|=m.
\end{equation*}
The active negative gradient set at $\x_0$ can be written as
    $$\tilde{\pp} = \left\{ \bfs : \begin{array}{rl}
        \sigma_i &= -s_i \quad i \in N, \\
        \sigma_i&=\pm 1 \quad i \in Z
        \end{array}\right\}.$$
Since $\|\bfs\|=\sqrt{n}$, the normalized active negative gradient set at $\x_0$ can be written as
    $$\pp = \left\{ \pee : \begin{array}{rl}
        p_i &= -s_i/\sqrt{n} \quad i \in N, \\
        p_i&=\pm 1/\sqrt{n} \quad i \in Z
        \end{array}\right\}.$$
Thus, the SOCP in Algorithm \ref{alg:maxFunction} becomes
$$\begin{array}{rcl}
&&\displaystyle\max_{t,\;\uu} \{t:\|\uu\| \;\le\; 1, \pee^\top\uu \;\ge\; t
    \quad \forall\,\pee\in\mathcal{P} \; \}\\
    &=&\displaystyle \max_{t,\;\uu} \left\{t:\|\uu\| \;\le\; 1,
    -\sum_{i \in N} s_i u_i - \sum_{i \in Z} |u_i|
    \;\ge\; \sqrt{n} t \right\} \\
    &=&\displaystyle \max_{t,\;\uu} \left\{t:\|\uu\| \;\le\; 1,
    -\sum_{i \in N} s_i u_i 
    \;\ge\; \sqrt{n} t, \quad u_i =0  \quad \forall i \; \in Z\right\}. \\
\end{array}$$
This is solved when  
    $$u^*_i = \left\{\begin{array}{rl}
    \dfrac{-s_i}{\sqrt{n-m}} & \quad\mbox{if } i \in N \\
    0 & \quad \mbox{if } i \in Z
    \end{array}\right\}
    \qquad \mbox{and} \qquad
    t^* = \frac{\sqrt{n-m}}{\sqrt{n}}.$$
Applying Step 2 of Algorithm \ref{alg:maxFunction}, we find
    $$\CM(F;\x_0) = \sqrt{1-(t^*)^2} = \sqrt{\frac{m}{n}} 
    \quad \mbox{and} \quad
    \CV(F;\x_0) =
    \left\{
        \begin{array}{rl}
             \Sn   &\mbox{if } m=n \\
            \dfrac{-\bfs}{\sqrt{n-m}} &\mbox{if } m<n,
        \end{array}
    \right\}.$$
   \qed
\end{proof}

Combining Theorem \ref{thm:cosineMeasureEllOne} with Theorem \ref{thm:valueofcmf}, we gain insight regarding the cosine measure required to ensure a set contains a descent direction of the $\ell_1$ norm at a given point.

\begin{corollary}
   Let $F:\R^n \to \R: \x \mapsto \Vert \x\Vert_1, n \geq 2$.  Suppose $\x_0$ has $m$ zero entries.  If $E$ is a set of vectors in $\R^n$ with cosine measure $\CM(E)>\frac{\sqrt{m}}{\sqrt{n}}$, then $E$ contains a descent direction for $F$ at $\x_0$.  (If $m=n$, this is impossible.)
\end{corollary}

\section{Conclusion} \label{sec:conclusion}

Several future research directions have emerged from this work.  An analysis of the finite-max function using  approximate gradients  to find the cosine measure value could be explored. In practice, the exact gradients are usually unknown and using  approximation of the gradients might be necessary. Gradients can be approximated using \emph{generalized simplex gradients} \cite{hare2020calculus,hare2020error} and bounds between the exact cosine measure  and the approximate cosine measure could be developed. Another obvious research directions is to investigate the cosine measure of other classical nonsmooth functions. For instance, the $\ell_1$ norm function of $A\x-\mathbf{b},$ the maximum eigenvalue function,  the nuclear norm function and the Hinge loss are reasonable choices.  

The cardinality of active set of vectors for a sine vector (or a cosine vector) is still unclear.  It is known that the active set may not  contain a basis of $\R^n$ in general. Depending on the properties of the set $D^c$,  it might be possible to develop results on the cardinality of the active set.

Another important topic to investigate is  the cosine measure of a set $\pspan(\Set)$  such that $\pspan(\Set)$ has an infinite frame. In this paper, we have assumed that all frames are finite.  However, this is not always the case.  One of the most famous examples of an infinite frame is  the \emph{Lorentz cone}: in $\R^3$, $\mathcal{L}=\pspan(\Set)$ where $\Set=\{ \vv \in \R^3: \vv=\bbm \cos(\theta)& \sin(\theta)&1\ebm^\top, \theta \in [0,2\pi) \}].$ In some cases, it might be possible  to develop results related to cosine measure.  Beyond such convex cones, nonconvex semialgebraic sets with explicitly characterized geometric structure provide another promising source of infinite directional families. In particular, bilinear constraint sets, including hyperbolas, possess continuously varying tangent and normal directions that could naturally be studied from the perspective of cosine measures \cite{bauschke2023projections,lal2023nonconvex}. These examples may help clarify when cosine measures associated with infinite families are computable and how the geometry of the underlying set influences their values and active directions.

\medskip

\noindent \textbf{Acknowledgments} The authors would like to thank Manish Krishan Lal for  his valuable comments and  helpful suggestions.  The authours would also like to thank  Dana Abukhalaf for her assistance in  proving Proposition \ref{prop:cvEqualPolar}.

\bibliographystyle{siam}
\bibliography{bibliography}
\end{document}